\documentclass[reqno,12pt]{amsart}

\usepackage{amssymb,amsmath,psfrag}
\usepackage{enumerate}
\usepackage{graphicx}
\usepackage{hyperref}
\usepackage{mathrsfs}

\newcommand\sectionpage{}
\renewcommand\sectionpage{\newpage}

\newtheorem{lem}{Lemma}[section]
\newtheorem{cor}[lem]{Corollary}
\newtheorem{prop}[lem]{Proposition}
\newtheorem{thm}[lem]{Theorem}

\theoremstyle{definition}

\newtheorem{exam}[lem]{Example}
\newtheorem{construction}[lem]{Construction}
\newtheorem{problem}[lem]{Problem}

\numberwithin{equation}{section}
\numberwithin{table}{section}
\numberwithin{figure}{section}

\newcommand\bbQ{\mathbb{Q}}
\newcommand\bbR{\mathbb{R}}
\newcommand\bbZ{\mathbb{Z}}

\newcommand\A{\mathrm{A}}

\newcommand\cI{{\mathcal I}}

\newcommand\cL{{\mathcal L}}

\newcommand\sN{\mathcal{N}}
\newcommand\cO{{\mathcal O}}
\newcommand\cP{{\mathcal P}}

\newcommand\cS{{\mathcal S}}

\renewcommand\ell{l}
\newcommand\bF{\mathbf{F}}
\newcommand\bV{\mathbf{V}}
\newcommand\PR{\ensuremath{\mathbb{PR}}}
\newcommand\PP{\Pi}

\renewcommand\setminus\smallsetminus

\begin{document}
\pagestyle{myheadings}
\markboth{\sc Fl\'orez and Zaslavsky}{\sc Projective Rectangles}


\vspace*{-1.25cm}

\title[Projective Rectangles]{Projective Rectangles: \\A New Kind of Incidence Structure}

\author{Rigoberto Fl\'orez}
\thanks{Fl\'orez's research was partially supported by a grant from The Citadel Foundation.}
\address{Dept.\ of Mathematical Sciences, The Citadel, Charleston, South Carolina 29409}
\email{\tt rigo.florez@citadel.edu}

\author{Thomas Zaslavsky}
\address{Dept.\ of Mathematics and Statistics, Binghamton University, Binghamton, New York 13902-6000}
\email{\tt zaslav@math.binghamton.edu}

\date{\today}

\begin{abstract}
A projective rectangle is like a projective plane that has different lengths in two directions.  We develop the basic theory of projective rectangles including incidence properties, projective subplanes, configuration counts, a partial Desargues's theorem, a construction, and alternative formulations.  In sequels we study harmonic conjugation and the graphs of lines and subplanes.
\end{abstract}

\subjclass[2010]{Primary 51A99; Secondary 05B15, 05B35, 05C22, 51A45, 51E26}

\keywords{Projective rectangle; incidence geometry; Pasch axiom; projective plane; Desargues's theorem; orthogonal array}

\maketitle

\vspace{-1.5cm}

\setcounter{tocdepth}{3}
\tableofcontents

\sectionpage\section {Introduction}

A projective rectangle is like a projective plane, but narrower than it is tall.  More precisely, it is like the set of points on a certain kind of family of lines in a projective plane, with their induced lines.  Very precisely, it is an axiomatic incidence structure based on adapting axioms of projective geometry, inspired by harmonic conjugation in matroids.

Projective rectangles, regarded as rank-3 matroids, are found in all known harmonic matroids, such as full algebraic matroids.  Harmonic matroids are matroids within which there is harmonic conjugation \cite{rfhc}; their definition was inspired by Lindstr\"om's article \cite{lhc} about abstract harmonic conjugation.  Harmonic conjugation applied to complete lift matroids of group expansions \cite[Example 6.7]{b1} of a triangle (for instance, $L_2^k$, Example \ref{ex:L2k}) led us to structures that looked like vertical strips in projective planes---whence the name ``projective rectangle'' and the impulse to find a general theory of this idea in terms of incidence geometry.  Projective rectangles themselves are almost examples of harmonic matroids, seemingly falling short only in special lines, as we prove in a sequel \cite{pr3}.

An indication of what we accomplish in this article:  First, the axioms (Section \ref{sec:pr}) and basic consequences for incidence geometry (Section \ref{sec:properties}) and counting (Section \ref{sec:finite}).  The crucial axiom (A\ref{Axiom:A6}), a restricted Pasch axiom, is the heart of projective rectangles with powerful consequences.  
Especially, we see that a projective rectangle, if it is not a projective plane, contains a multitude of maximal projective planes; we call them its ``planes''.  Section \ref{sec:desargues} develops partial Desarguesian properties of projective rectangles, which satisfy limited versions of the two halves of Desargues's Theorem.  In Section \ref{sec:subplane} we show that the construction based on a subplane and a special point, alluded to above, actually works to produce projective rectangles in planes that are Pappian, i.e., coordinatized by a field; we do not know how far that subplane construction generalizes.  The following section treats the narrowest projective rectangles, which are the simplest and best understood.  Next are two sections that give alternative viewpoints: in Section \ref{S:OA} we see that a projective rectangle is essentially a Paschian transversal design and thus is equivalent to a special kind of orthogonal array, and in Section \ref{sec:dual} we take the approach of projective duality by interchanging points and lines, which may suggest new properties but which we have not studied deeply.  
We have only an elementary understanding of projective rectangles in general, as is shown by the list of significant open problems in Section \ref{sec:open}.

In sequels we treat adjacency graphs and harmonic conjugation.  
The sequel \cite{pr3} explores abstract harmonic conjugation as a theme linking harmonic matroids and projective rectangles.  In one direction, a projective rectangle is almost a harmonic matroid.  In the other direction, a harmonic matroid contains a projective rectangle if it contains a matroid of a finite-field expansion of a triangle, in particular if it contains a Reid cycle matroid.  There we prove by harmonic conjugation that we use here (Section \ref{sec:subplane}) to establish validity in many situations of the subplane construction.  
(Harmonic conjugation provides a recursive construction from which it is clear that the result is a projective rectangle.  In contrast, the subplane construction involves restricting the lines and points based on using a given subplane of the main plane; the validity of the Pasch axiom (A\ref{Axiom:A6}) is not obvious.)

Two other sequels explore the graphs of adjacency of (short) lines and of planes in finite projective rectangles \cite{pr2a, pr4}.  The graph of lines, where adjacency means having a point in common, is a known strongly regular graph and by using its properties we prove the validity of the subplane construction in finite Desarguesian planes.  In projective rectangles that are not projective planes the graph of planes, where adjacency means having a short line in common, has striking internal structure that presents a tantalizing vision of higher dimensionality.  

Our personal interest is mainly in finite systems, but many results apply to infinite projective rectangles.  For instance, Section \ref{sec:properties} encompasses infinite systems, while Section \ref{sec:finite} requires finiteness.  Our viewpoint is influenced by matroid theory but is largely that of incidence geometry; matroid theory is not needed to read this paper.

We wish to acknowledge the inspiration of the elegant and deep short papers \cite{ldt, lhc} of Bernt Lindstr\"om.   Lindstr\"om's ideas, as further developed by the first author in his doctoral dissertation and \cite{rflc, rfhc}, led to this study of projective rectangles.

\subsection*{Acknowledgement}\

We are grateful to a referee of a previous version, who read the paper with great care and warned us of errors, gaps, and infelicities.  Those observations helped lead us to major improvements.

\sectionpage\section {Projective rectangles}\label{sec:pr}

An \emph{incidence structure} is a triple $(\cP,\mathcal{L},\mathcal{I})$ of sets with $\mathcal{I}
\subseteq \cP \times \mathcal{L}$. The elements of $\cP$
are \emph{points}, the elements of $\mathcal{L}$ are \emph{lines}.
A point $p$ and a line $l$ are \emph{incident} if $(p,l) \in \mathcal{I}$. A set $P$ of points is said to be
\emph{collinear} if all points in $P$ are in the same line. We say that two
distinct lines \emph{intersect in a point} if they are incident with the same point.

A \emph{projective rectangle} is an incidence structure $(\cP,\mathcal{L},\mathcal{I})$ that satisfies the following axioms:

\begin{enumerate} [({A}1)]
\item \label{Axiom:A1}   Every two distinct points are incident with exactly one line.

\medskip

\item \label{Axiom:A2} There exist four points with no three of them collinear. 

\medskip

\item \label{Axiom:A3}  Every line is incident with at least three distinct points. 

\medskip

\item \label{Axiom:A4}  There is a \emph{special point} $D$.
A line incident with $D$ is called \emph{special}.  A line that is not incident with $D$ is called \emph{ordinary}, and a point that is not $D$ is called \emph{ordinary}.

\medskip

\item \label{Axiom:A5}  Each special line intersects every other line in exactly one point.

\medskip

\item \label{Axiom:A6}  Let $l_1$ and $l_2$ be two ordinary lines that intersect in a point. If $l_3$ and $l_4$ are distinct lines that intersect $l_1$ and $l_2$ in four distinct points, then $l_3$ and $l_4$ intersect in a point.  (We sometimes will call $l_1$ and $l_2$ the ``intersecting lines'', and $l_3$ and $l_4$ the ``crossing lines'' because they cross the first two.)

\end{enumerate}

A \emph{complete quadrilateral} is an incidence structure that consists of four lines, no three concurrent, and their six points of intersection.  A \emph{nearly complete quadrilateral} is like a complete quadrilateral but with only five of the intersection points; the sixth intersection point may or may not exist.  Axiom (A\ref{Axiom:A6}) states that almost every nearly complete quadrilateral in a projective rectangle is complete.  This is a partial Pasch axiom (e.g., see \cite[page 314]{vw}), not the full Pasch axiom because it has an exception when either of the first two lines is special; then the remaining two lines may or may not be concurrent.  This exception is what admits projective rectangles that are not projective planes.  Section \ref{S:OA} has more discussion of the significance of Axiom (A\ref{Axiom:A6}).

\emph{Notation}: $\PR$ denotes a projective rectangle.  Axiom (A\ref{Axiom:A3}) lets us treat lines as sets of points; thus for a point $p$ and line $l$, the statements ``$p$ is a point of $l$'', ``$p \in l$'', ``$p$ is on $l$'', ``$l$ contains $p$'', etc., mean the same.  
We write $\overline{pq}$ for the unique line that contains two points $p$ and $q$.  After we establish the existence of projective planes in \PR, we use the notation $\overline{abc\dots}$ to mean the unique line (if $abc\dots$ are collinear) or plane (if they are coplanar but not collinear) that contains the points $abc\dots$.

The projective planes are  some familiar examples of projective rectangles. 
A projective plane is called a \emph{trivial projective rectangle}.  
In particular the Fano plane $F_7$ is the smallest projective rectangle (see Theorem \ref{numberofpoinsinlines} Part \eqref{numberofpoinsinlines:a}).
The non-Fano configuration is not a projective rectangle; it fails Axiom (A\ref{Axiom:A6}).

In a trivial projective rectangle $\PR$ the special point $D$, although selected as part of the definition, may be chosen to be any point.  We show in Theorem \ref{numberofpoinsinlines} that every special line has the same number $n$ of ordinary points and every ordinary line has the same number $m+1$ of points, and $m\leq n$.  If $m\neq n$, then the special lines are distinguished from the ordinary lines by cardinality; in that case $D$ is determined by lying on all special lines.  If $m<n$, therefore, the special point is unique.  In general we regard the specification of $D$ as part of the definition of a projective rectangle.

A projective rectangle $\PR$ is a rank-3 matroid; the elements are the points of $\PR$ and the rank-2 flats are the lines of $\PR$.  (The matroid is, of course, infinite if the number of points in a special line is not finite.)  The simplest such matroid is that of the following example.

\begin{exam}\label{ex:L2k}
The matroid $L_2^k$ is another example of a projective rectangle
(see Figure \ref{figure1}).  It has $m+1=3$ special lines.  Let $A:= \left\{ a_g \mid g \in \bbZ_2^k \right\}
\cup \{D \}$, $B:= \left\{ b_g \mid g \in \bbZ_2^k \right\} \cup \{D \}$ and
$C:= \left\{ c_g \mid g \in \bbZ_2^k \right\} \cup \{D \}$.
Let $L_2^k$ be the simple matroid of rank 3 defined on the ground set
$E:= A\cup B\cup C$ by its rank-2 flats, which are the special lines $A$, $B$, $C$ and the sets $\{a_g, b_{g+ h}, c_h \}$ with $g$ and $h$ in $\bbZ_2^k$, which are the ordinary lines.

We note that $L_2^k$ is the complete lift matroid of the group expansion of a triangle, i.e., $L_0(\bbZ_2^k)$ in the language of \cite{b1, b2}.
We say more about projective rectangles with $m=2$ in Section \ref{sec:narrow}.

\begin{figure} [htbp]
\begin{center}
\includegraphics[width=8cm]{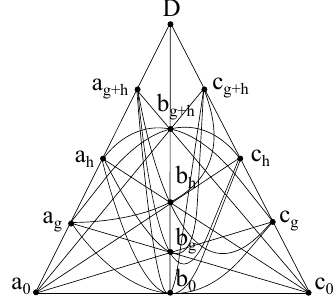}
\caption{The matroid $L_2^2$ with group the Klein 4-group, $\mathfrak{V}_4 = \{1,g,h,g+h\} \cong \bbZ_2 \times \bbZ_2$.} \label{figure1}
\end{center}
\end{figure}
\end{exam}

\sectionpage\section{Properties of projective rectangles}\label{sec:properties}

In this section we study essential properties of projective rectangles.  We begin with basic facts; then we prove that the projective rectangle contains projective planes and we conclude with a section of counting formulas for later use.

\subsection{Fundamental properties}\label{sec:fundamental}\

If a projective rectangle $\PR$ with exactly $m+1$ special lines has one of them with $n+1$ points, then we say that the \emph{order} of $\PR$ is $(m,n)$.  We do not assume $m$ or $n$ is finite unless we so state.  
In Theorem \ref{numberofpoinsinlines} we prove $m\le n$; we also prove that every special line has the same number of points, that every ordinary line has the same number of points, and many other elementary facts about points and lines.

The following result states basic properties of a projective rectangle.

\begin{thm} \label{numberofpoinsinlines} 
If\/ $\PR$ is a projective rectangle of order $(m,n)$, then the following hold in $\PR$:
\begin{enumerate}[{\rm(a)}]

\item  \label{partitionPR:i} The point set of $\PR\setminus D$ is partitioned by all special lines deleting $D$.

\item \label{numberofpoinsinlines:a} There are at least three special lines
and four ordinary lines.  Moreover, there are at least seven points.

\item  \label{numberofpoinsinlines:b} If $l$ is a line and $p$ is a point not in $l$, then the number of distinct
lines incident with $p$ intersecting $l$ equals the number of points on $l$.

\item  \label{numberofpoinsinlines:da} 
Through each ordinary point there passes exactly one special line.

\item  \label{numberofpoinsinlines:d} 
All ordinary lines have the same number of points.  The number of points in an ordinary line is equal to the number of special lines, that is, $m+1$. 

\item  \label{numberofpoinsinlines:c} All special lines have the same number of points, i.e., $n+1$ points, and the same number of ordinary points, i.e., $n$.

\item \label{numberofpoints}  There are exactly $(m+1)n$ ordinary points.

\item  \label{numberofpoinsinlines:e} The number of lines incident with an
ordinary point is equal to the number of points in a special line, that is, $n+1$.
The number of ordinary lines that contain each ordinary point is $n$.

\item  \label{numberofpoinsinlines:g} The number of points in a special line
is at least the number of points in an ordinary line; that is, $n \geq m$.

\item  \label{numberofpoinsinlines:h}  There are exactly $n^2$ ordinary lines.

\item \label{cor:numberofpoinsinlines:b} For a given point $p$ in an ordinary line $l$,
there are $n-1$ ordinary lines intersecting $l$ at $p$.

\end{enumerate}
\end{thm}

\begin{proof}
Proof of Part (\ref {partitionPR:i}). 
By Axiom (A\ref{Axiom:A1}), every point $p \in \PR\setminus D$ belongs to the unique special line $\overline{pD}$.

Proof of Part (\ref{numberofpoinsinlines:a}). From Axiom (A\ref{Axiom:A2}) we know that in $\PR$ there are four points, no three of them collinear. If one is $D$, each other one with $D$ generates a special line, all of which are distinct by noncollinearity.  If none of them is $D$, the points generate six distinct lines, of which at most two can contain $D$ because no three of the four points are collinear.  
Thus, the four remaining lines are ordinary lines. Since in one of the ordinary lines there are at least three points, these points form with $D$ three special lines. 
We have proved that in $\PR$ there are at least three special lines and three ordinary lines.  
By Axiom (A\ref{Axiom:A3}), each special line contains at least two ordinary points, so there are at least seven points.

Now consider two special lines $s, s'$ and two ordinary points $p_1,p_2$ on $s$ and $p_1',p_2'$ on $s'$.  The lines $\overline{p_ip'_j}$ are four distinct ordinary lines.

We prove Part (\ref{numberofpoinsinlines:b}). 
Let $q \in l$ and $p\notin l$.  From (A\ref{Axiom:A1}) there is exactly one
line incident with $p$ that intersects $l$ at $q$, and all such lines are distinct.

We prove Parts (\ref{numberofpoinsinlines:da}) and (\ref{numberofpoinsinlines:d}).  
Given an arbitrary ordinary line $l$, we know by (A\ref{Axiom:A1}) that each point in $l$ together with $D$ determines a unique special line.
Every special line is generated in this way, by (A\ref{Axiom:A5}).
Thus, there is a bijection between the special lines and the points in $l$.  This implies the number of points in any ordinary line equals the number of special lines.

We prove Parts (\ref{numberofpoinsinlines:c}) and  (\ref{numberofpoinsinlines:e}). 
We suppose that $l_1$ and $l_2$ are special lines in $\PR$ with $n_1+1$ and $n_2+1$
points, respectively.  Let $p$ be a point non-incident with
either of those lines. Part \eqref{numberofpoinsinlines:b} implies
that there are $n_1+1$ distinct lines intersecting $l_1$ that are incident with $p$.
Those $n_1+1$ lines also intersect $l_2$. Indeed, one of those lines is  special and
the remaining $n_1$ lines intersect $l_2$ because they are ordinary.
Therefore, $n_1 \leq n_2$. Similarly, $n_2 \leq n_1$. This proves that all special lines have the same number of points.  Deducting 1 for the special point $D$ gives the number of ordinary points on a special line.

Proof of Part (\ref{numberofpoints}).  The number of special lines is $m+1$, Part \eqref{numberofpoinsinlines:c} says the number of ordinary points in each special line equals $n$ and Part \eqref{partitionPR:i} says the special lines partition the ordinary points.

Proof of Part (\ref{numberofpoinsinlines:g}). 
We suppose that $l$ is an ordinary line and $s$ is a special line.  We produce an injection of the point set of $l$ into the point set of $s$.  Let $p$ be a point not in either $l$ or $s$; it exists because any special line other than $s$ contains at least three points by Axiom (A\ref{Axiom:A3}), one of which is not $D$ and not in $l$.  For each $q \in l$, the line $\overline{pq}$ intersects $s$ in a point $r_q$ by Axiom (A\ref{Axiom:A5}).  For distinct $q,q' \in l$, $r_q$ and $r_{q'}$ must be distinct because the lines $\overline{pr_q}=\overline{pq}$ and $\overline{pr_{q'}} = \overline{pq'}$ are distinct and intersect only at $p$ by Axiom (A\ref{Axiom:A1}).  Thus, the mapping $l \to s$ by $q \mapsto r_q$ is an injection.  This proves that $s$ has at least as many points as $l$, i.e., $n+1 \geq m+1$.

By Part \eqref{numberofpoinsinlines:c} every special line has $n$ ordinary points, and by definition there are $m+1$ special lines.

Proof of Part (\ref{numberofpoinsinlines:h}). Let $p$ be a point in an ordinary line. Two ordinary points in two special lines give rise to a unique ordinary line. Since every special line has $n+1$ points and one of them is $D$, it is easy to see that the two special lines give rise to $n^2$ ordinary lines. Those are all the ordinary lines that intersect the two special lines. Since every ordinary line intersects every special line, we conclude that there are no more ordinary lines in $\PR$.

Proof of Part (\ref{cor:numberofpoinsinlines:b}). Since $p$ is a point in an ordinary line $l$,
from Part (\ref{numberofpoinsinlines:e})  there are $n+1$
lines incident with $p$. Only one of those $n+1$ lines is special; the other $n$ are not. 
This implies that there are $n-1$ ordinary lines intersecting $l$ at $p$.
\end{proof}

\subsection{Projective subplanes}\label{sec:planes}\

We show that a projective rectangle is a combination of projective planes, in the strong sense that every two intersecting ordinary lines are lines of a substructure that is a projective plane.  Before our results, though, we have to clarify the notion of substructure of an incidence structure $(\cP,\cL,\cI)$.

An \emph{incidence substructure} of $(\cP,\cL,\cI)$ is an incidence structure $(\cP',\cL',\cI')$ in which $\cP' \subseteq \cP$, $\cL' \subseteq \cL$, and $\cI' = \cI|\cP'\times\cL'$, i.e., the incidence relation is the same as in the superstructure but restricted to the elements of the substructure.  In particular, if $(\cP',\cL',\cI')$ is a projective plane, we call it a \emph{subplane} of $(\cP,\cL,\cI)$.

In a projective rectangle a subplane may contain an ordinary line and all its points; we call that kind \emph{full}.  A full subplane necessarily has order $m$.  A subplane need not be full; it also need not be a maximal subplane, for instance if it is a proper subplane of a full subplane.  In fact, that is the only way a subplane can fail to be maximal, as we will see in Theorem \ref{prop:maximalsubplane}.

The special point $D$ is very special, as are the special lines.

\begin{prop}\label{prop:Dinpp}
In a projective rectangle $\PR$, the special point $D$ is a point of every full subplane.  Also, for every special line $s$ and every full subplane $\pi$, $s\cap\pi$ is a line of $\pi$.
\end{prop}

\begin{proof}
A full subplane $\pi$ contains at least two lines, $l$ and $l'$, which intersect at a point $p \in \pi$, and at least one is ordinary, say $l$.  

If $l'$ is ordinary, then every special line $s$ intersects both $l$ and $l'$ at different points, unless $s$ is the special line $s_p$ on $p$.  These two points of $s$ determine a line of $\pi$, which is the intersection of $s$ with $\pi$.  Thus, for every special line except possibly $s_p$, $s \cap \pi$ is a line of $\pi$.   

If $l'$ is special, or rather if $l'=s'\cap\pi$ for some special line $s'$, then there is at least one point $p'$ on $l'$ that is neither $p$ nor $D$.  Let $q$ be a point in $l \setminus p$; then $\pi$ has a line $l''$ determined by $p'$ and $q$, which is ordinary since it contains not only $p \in s_p$ but also $q \notin s_p$.  Then we can replace $l'$ by $l''$ and have the case of two ordinary lines, so we may as well assume $l'$ is ordinary.

Let $s_1$ and $s_2$ be two special lines that are not $s_p$.  Then $s_1 \cap \pi$ and $s_2 \cap \pi$ are lines of $\pi$ whose intersection is a point $d \in \pi$.  Since $\{d\} = (s_1 \cap \pi) \cap (s_2 \cap \pi) = (s_1 \cap s_2) \cap \pi = \{D\} \cap \pi$, we conclude that $D = d \in \pi$.

Let $p_1$ be the intersection of $l$ with $s_1$ and let $p_2$ be the intersection of $l'$ with $s_2$.  Since $p_1 \notin l'$ and $p_2 \notin l$, the line $l''$ of $\pi$ determined by $p_1$ and $p_2$ does not contain $p$.  Since the points $p_1,p_2$ are not $D$ and are not in the same special line, $l''$ is ordinary, hence it is contained in $\pi$.  Being ordinary, by Axiom (A\ref{Axiom:A5}) $l''$ intersects $s_p$ in a point $p_{12}$, which cannot be $p$, so $p$ and $p_{12}$ determine a line of $\pi$, which must be $s_p\cap\pi$.  That is, $s_p\cap\pi$ is a line of $\pi$.
\end{proof}

Now we present the fundamental result about subplanes.

\begin{thm}[Planes in \PR]\label{prop:twolinesintersectingpp} Let $\PR$ be a projective rectangle. If two ordinary lines in $\PR$ intersect in a point, then both lines are lines of a unique full projective plane in $\PR$.
\end{thm}

First we state the construction that gives the projective plane.

\begin{construction}\label{Definition:Projective:Plane}
	Let $l_0$ and $l_1$ be ordinary lines in $\PR$ with exactly one point $q$ in common.  (See Figure \ref{fig:planeontwolinesintersecting}.)  
	Let $a_{0s} \in  l_{0}\cap s$ and $a_{1s} \in  l_{1}\cap s$, where $s$ ranges over the set $\cS$ of special lines in $\PR$, and pick three special lines to be called $x$, $y$, and $z$ such that $q \in x$.  Thus, $q=a_{0x}=a_{1x}$.  
    (We know there are three special lines by Theorem \ref{numberofpoinsinlines} Part (\ref{numberofpoinsinlines:a}).) 
    Let $b_{1s} \in  n_{1}\cap s$, where $n_{1}$ is the ordinary line that passes through $a_{0y}$ and $a_{1z}$.
	
\begin{figure}[ht]
\includegraphics[width=10cm]{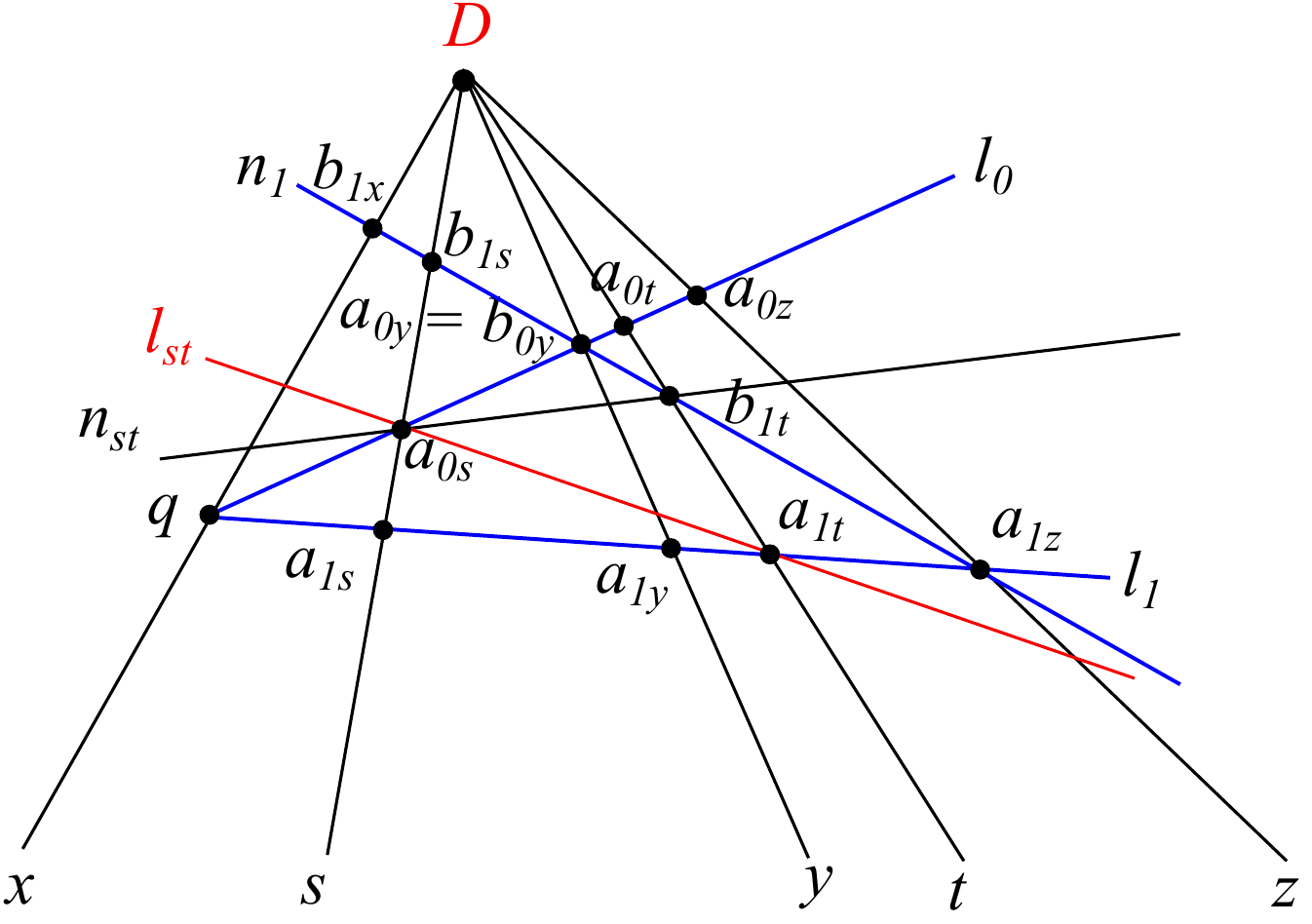}
\caption{Construction \ref{Definition:Projective:Plane}.}
\label{fig:planeontwolinesintersecting}
\end{figure}

   	Suppose that $s$ and $t$ denote two special lines. We  denote by $l_{st}$ the ordinary line passing
   through $a_{0s}$ and $a_{1t}$ with $s,t \ne x$  and we denote  by $n_{st}$ the
   ordinary  line passing through $a_{0s}$ and $b_{1t}$ with $s,t\ne y$. Let
   	$$L=\{l_{st}: s,t \in \cS, \ s,t \ne x \text{ and } s\ne t \}$$ 
	and
   	$$N=\{n_{st}: s,t \in \cS, \ s,t \ne y \text{ and } s\ne t \}.$$
	Note that $n_{1} = l_{yz} \in L$ and $l_{1} = n_{xz} \in N$, but $l_0 = n_{xy} \notin N$.
		
	We set $\PP:=(\cP_\PP,\mathcal{L}_\PP,\mathcal{I}_\PP)$, where $\mathcal{I}_\PP$ is
the incidence relation defined in $\PR$ and
	\[
	\begin{array}{lcl}
	\cP_\PP	& := &	{ (\bigcup_{l\in N} l) \cup (\bigcup_{l\in L} l) \cup l_0 \cup \{ D \}},\\[3pt]
	\cL_1 		& := &  	\left\{ s\cap\cP_\PP : s \in \cS \right\}, \\[3pt]
	\cL_2		& := &  	L \cup N \cup \{ l_0\}, \\[3pt]
	\cL_\PP 		& := & 	 \cL_1 \cup \cL_2.
	\end{array}
	\]
\end{construction}

\begin{proof}["New" New Proof of Theorem \ref{prop:twolinesintersectingpp}]  
We begin with the incidence structure  $\PP$ given by Construction \ref{Definition:Projective:Plane}. With the notation there, we prove that $\PP$ is a projective plane. 

First of all, we note that one of the defining properties of a projective plane, that there are four points in $\cP_\PP$ with  no three of them collinear, is satisfied by $a_{0y}$, $a_{1z}$, $q$, and $D$.

\begin{figure}[ht]
\includegraphics[width=7cm]{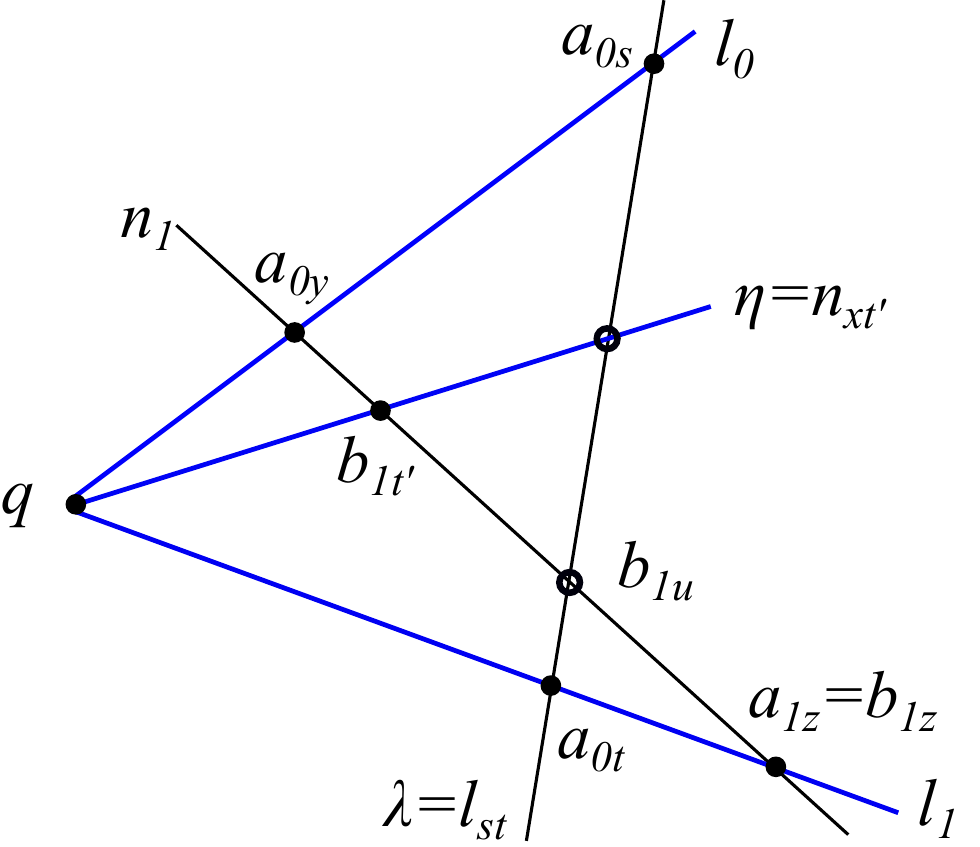}
\caption{For the proof of Theorem \ref{prop:twolinesintersectingpp}: lines $\lambda=l_{st}$ and $\eta=n_{xt'}
$ intersect (general case).}
\label{fig:twolinesintersectingproof}
\end{figure}

We next prove that given two lines in $\PP$, they intersect.  We begin with a lemma that simplifies the list of lines in $\PP$.

\begin{lem}\label{lemma0}
Every line $n_{st}$ with $s, t \neq x$ is a line $l_{st'}$.  Each line $n_{xt}$ is $\overline{qb_{1t}}$.
\end{lem}

\begin{proof}
By definition of $n_{st}$ (with $s,t \neq x$), we have $s,t \neq y$.  The lines $l_1$ and $n_{st}$ are in $\PP$ so they intersect in a point $a_{1t'}$.  We know $a_{1t'} \neq q$ since $q \notin n_{st}$, therefore $t' \neq x$.  Because $a_{0s} \in n_{st}$ and $\notin l_1$, we know that $a_{1t'} \neq a_{0s}$, so $t' \neq s$ and the line $l_{st'}$ is well defined.

A line $n_{xt} = \overline{a_{0x}b_{1t}} = \overline{qb_{1t}}$.
\end{proof}

Suppose that the two given lines are in $L$ (so they are ordinary). If they intersect in a  point in $l_0$ or in a point in $l_1$, there is nothing to prove.  Suppose that neither of those two cases holds. So, they are two ordinary lines
that intersect $l_0$ and $l_1$ in four different points. Therefore, by Axiom (A\ref{Axiom:A6}) the two given lines  intersect. 
By a similar argument we conclude that  if the two given lines are in $N$, then they intersect. It is clear
that any two lines in $\cL_1$ intersect in $D$ and that a line in $\cL_2$ intersects every line in $\cL_1$.

Suppose the two given lines are  $\lambda$ and $\eta $ with $\lambda=l_{st} \in L$ and $\eta \in N$.  We may assume $\eta \neq l_1$ because $\lambda$ and $l_1$ intersect by the definition of $L$.  By definition, $\lambda = l_{st}$ for $s \neq t$ and both $\neq x$.  By Lemma \ref{lemma0} we may assume $\eta = n_{xt'}$ for some $t' \neq x, y, z$.

Suppose that $a_{0y} \neq a_{0s}$ and $a_{1t} \neq a_{1z}$. 
Since $l_{st}$ and $n_{1}$ intersect $l_0$ and $l_1$ in four distinct points, by (A\ref{Axiom:A6}) we know that $l_{st}$ intersects $n_{1}$ at a point $b_{1u}$.  Then $l_{st}$ and $n_{xt'}$ intersect $l_0$ and $n_{1}$ in four distinct points (because $n_{1}$ intersects $l_0$ at $a_{0y} \notin l_{st}$), unless $b_{1u} = b{1t'}$.  
That, with (A\ref{Axiom:A6}), implies that $l_{st}$ and $n_{xt'}$ intersect in a point.  In the exceptional case that $b_{1u} = b{1t'}$, that is the intersection point.

If $a_{0y} = a_{0s}$, then $l_{st}$ and $n_{xt'}$  intersect $l_0$ and $n_{1}$ in four distinct points unless $a_{1t} = a_{1z}$. Since $l_0$ and $n_{1}$ intersect in $a_{0y}$, by (A\ref{Axiom:A6}) we conclude that $l_{st}$ and $n_{xt'}$ intersect.  If $a_{1t} = a_{1z}$, $l_{st} = n_1$ and the intersection point is $b_{1t'}$.

In the remaining special case $a_{1t} = a_{1z}$ and $a_{0y} \neq a_{0s}$.  Here $l_{st}$ and $n_{xt'}$ intersect $l_0$ and $n_{1}$ in four distinct points so by (A\ref{Axiom:A6}) we obtain the desired intersection.

Finally, suppose one line is $l = s \cap \cP_\PP$ for a special line $s$.  The other is either $t \cap \cP_\PP$ for a special line $t$, and intersects $l$ at $D$, or is an ordinary line $l' = l_{st}$ or $n_{st}$, in which case it intersects $l$ at a point in $l' \subseteq \cP_\PP$.

This completes the proof that any two lines in $\PP$ intersect in $\PP$.

\begin{figure}[ht]
\includegraphics[width=8cm]{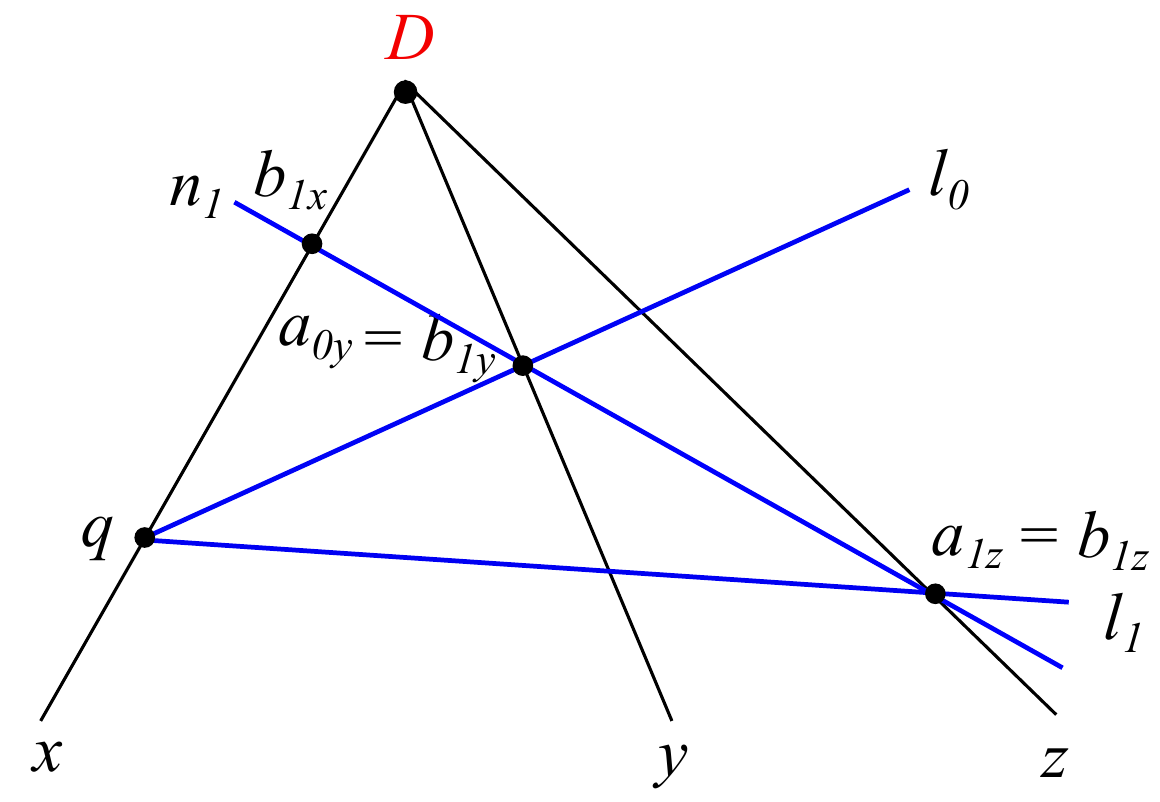}
\caption{For the proof of Theorem \ref{prop:twolinesintersectingpp}: the framework for intersections of ordinary lines of $\PP$, with the fundamental ordinary lines $l_0, l_1, n_1$ and important special lines $x, y, z$.}
\label{fig:twopointsgeneratingproof}
\end{figure}

We now prove that given two points $p_0, p_1 \in \cP_\PP$, they are in a line in $\PP$.  (If they are in one line, they cannot be in two, because the lines of $\PP$ are ordinary lines or restrictions of special lines of $\PR$, and every line in \PR\ is determined by two of its points.)  

Our objective is to prove that the line $\overline{p_0p_1}$ is either a line in $\PP$ or is special.  We divide the proof into cases depending on how $p_0$ and $p_1$ are located with respect to the fundamental lines $l_0, l_1, n_1$ of Construction \ref{Definition:Projective:Plane}.  The cases are:
\begin{enumerate}[{Case} 1.]
\item\label{ppspecial} $p_0,p_1$ are both in a special line or both in any one of $l_0$, $l_1$, or $n_1$.
\item\label{ppq} $p_0=q$.  By symmetry this covers the case $p_1 = q$.
\item\label{ppx} $p_0 \in x \setminus q$.  This covers the case $p_1 \in x \setminus q$.
\item\label{ppgeneric} The generic case, where there are no collinear threesomes among $q, p_0, p_1, D$ and neither point is in one of the fundamental lines $l_0, l_1, n_1$.  By the definition of $\cP_\PP$ and Lemma \ref{lemma0}, each of $p_0$ and $p_1$ must belong to an ordinary line $l_{st}$ or $n_{xt}$.
\item\label{ppl0} $p_0 \in l_0 \setminus q$ and $p_1 \notin x, l_0$.
\item\label{ppn1} $p_0 \in n_1$ but is not in any of $x,l_0,l_1$, and $p_1 \notin x, l_0, l_1, n_1$.
\end{enumerate}

Now we prove each case.  We provide figures for the more complicated cases and some subcases.


Case \ref{ppspecial}. If $p_0,p_1$ are both in a special line, its intersection with $\cP_\PP$ is in $\cL_1 \subseteq \cL_\Pi$.  If $p_0,p_1$ are both in any one of $l_0$, $l_1$, or $n_1$, then $\overline{p_0p_1}$ is in $\cL_\PP$ by definition.

Case \ref{ppq}. Suppose $p_0=q$.  By Case \ref{ppspecial} we may assume $p_1 \notin x \cup l_0 \cup l_1$.  
Now, $p_1$ is in a line $n_{xt}$ or $l_{st}$.  In the former case $\overline{p_0p_1} = n_{xt}$.  
In the latter case there is a point $b_{1t'}$ common to $l_{st}$ and $n_1$.  If $s \neq y$ (so $t' \neq y$), Axiom (A\ref{Axiom:A6}) with intersecting lines $l_0, l_{st}$ and crossing lines $n_1$ and $\overline{p_0p_1}$ produces an intersection point $b_{1t''}$; and if $s \neq z$ (so $t' \neq z$), lines $l_1, l_{st}$ and $n_1, \overline{p_0p_1}$ produce an intersection point $b_{1t''}$.  Either way, $\overline{p_0p_1} = n_{qx} \in \cL_\PP$ where $x$ is the special line that contains $b_{1t''}$.

Case \ref{ppx}. Suppose $p_0 \in x$ but is not $q$; then $p_0 \in l_{st}$ for some $s,t \neq x$.  We may assume $p_1 \notin x, l_{st}$; then $p_1 \in l_{s't'}$ or $n_{xt'}$.  

If the former, (A\ref{Axiom:A6}) with intersecting lines $l_{st}, l_{s't'}$ and crossing lines $l_0, \overline{p_0p_1}$ imply an intersection $a_{0,s''}$ if $a_{0s} \neq a_{0s'}$ (i.e., $s\neq s'$), and replacing $l_0$ by $l_1$ implies an intersection point $a_{1t''}$ if $a_{1t} \neq a_{1t'}$ (i.e., $t \neq t'$); note that one of these must hold or $p_0, p_1 \in l_{st}=l_{s't'}$ and we are done.  If both hold, then $\overline{p_0p_1} = l_{s''t''}$ and we are done.

If $s=s'$, then the intersecting lines $l_{st}, l_{s't'}$ and crossing lines $l_1, \overline{p_0p_1}$ imply a point $a_{1t''} \in l_1 \cap \overline{p_0p_1}$.  Then intersecting $l_1, l_{st}$ and crossing $l_0, \overline{p_0p_1}$ generate a point $a_{0s''} \in l_0 \cap \overline{p_0p_1}$.  Thus, $\overline{p_0p_1} = l_{s''t''}$ and we are done.  The case $t=t'$ is similar.

Case \ref{ppgeneric}.  We suppose $p_0,p_1 \notin l_0,l_1, n_1$, no three of $q, p_0, p_1, D$ are collinear, and there are ordinary lines $k_0,k_1$ in $\PP$ such that $p_0 \in k_0$ and $p_1 \in k_1$.
 Since $\overline{p_0p_1} \not\ni D$, it is an ordinary line.  
We show that $\overline{p_0p_1}$ is an $l_{st}$, hence a line of $\PP$.  First, $k_0$ and $k_1$ intersect because they are lines of $\PP$.  Let $r$ be their intersection.  

\begin{figure}[ht]
(a)\includegraphics[width=7cm]{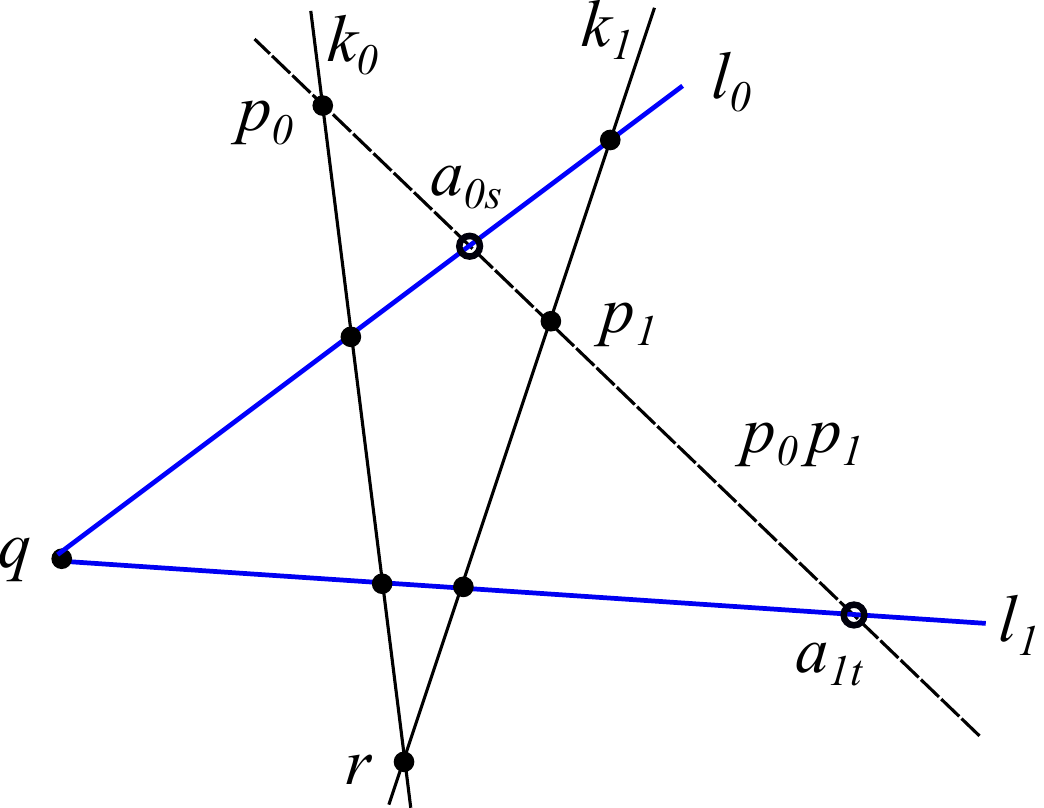}
\quad
(b)\includegraphics[width=7cm]{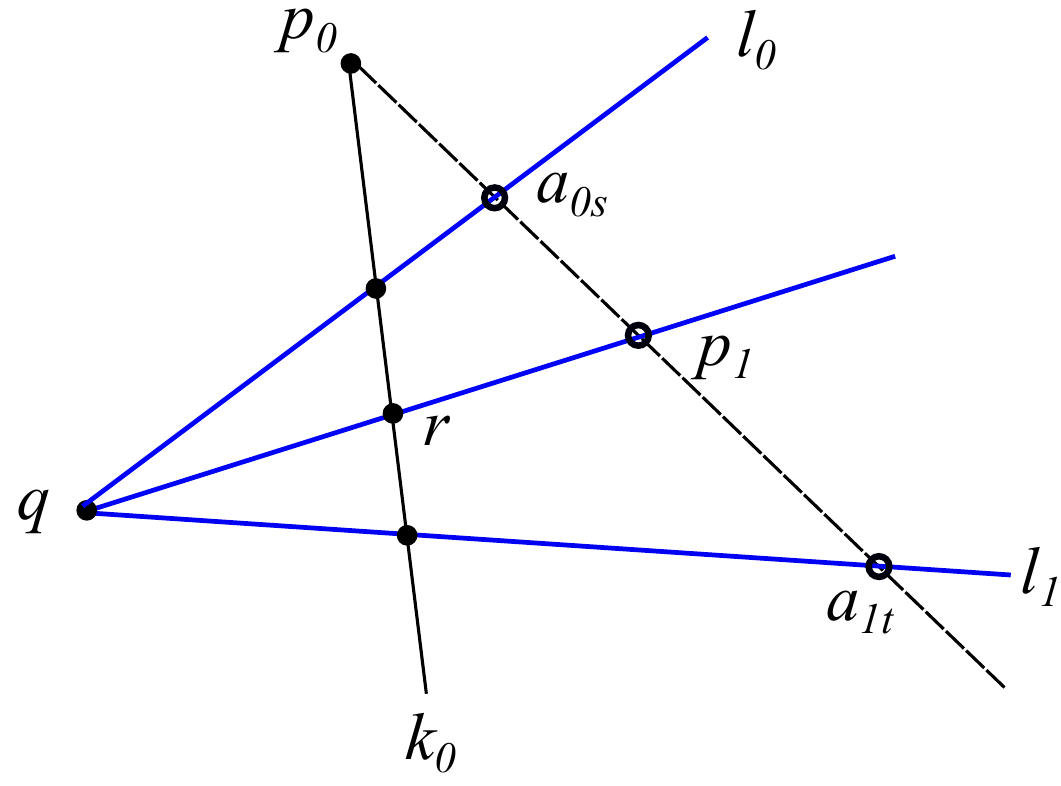}
\\\vspace{8mm}
(c)\includegraphics[width=6.5cm]{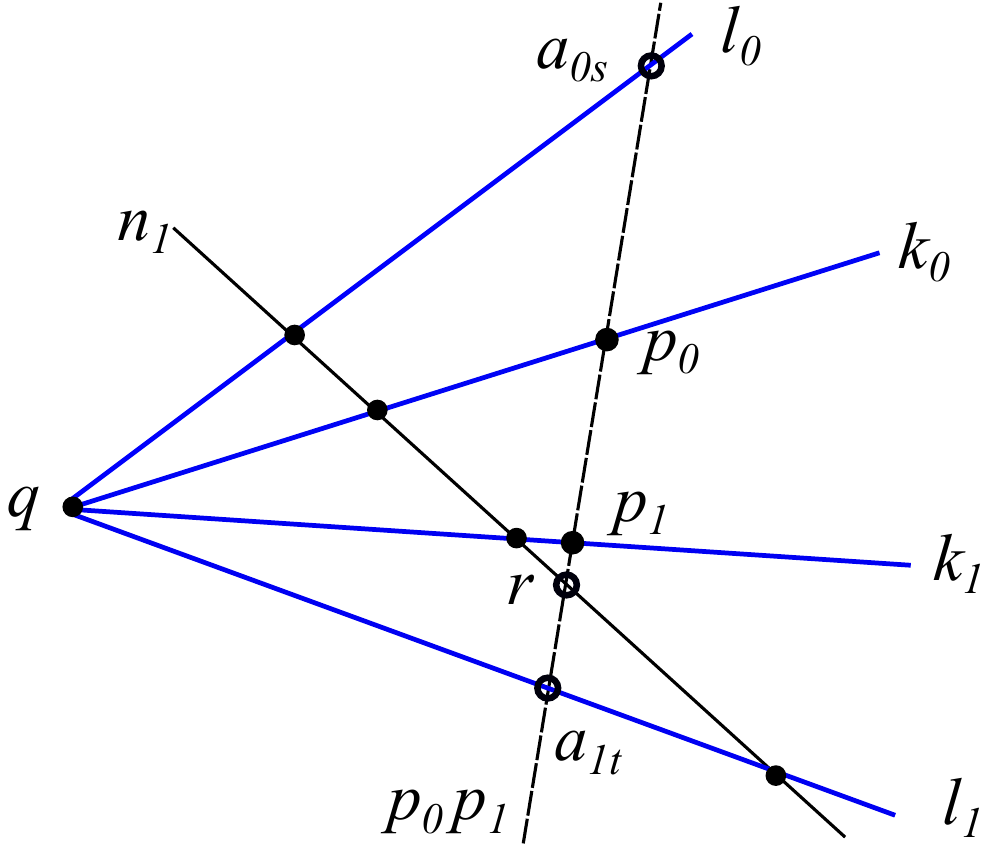}
\caption{For the proof of Case \ref{ppgeneric}, with the three possible types of pair $k_0, k_1$.}
\label{fig:twopointsgeneratingproof4}
\end{figure}

Case \ref{ppgeneric}(a).  Assume $k_0, k_1 \in L$.  There are two subcases.

If $r \notin l_0 \cup l_1$, then (for $i=0,1$) $l_i \cap \overline{p_0p_1}$ is a point, by Axiom (A\ref{Axiom:A6}) applied to the intersecting lines $k_0, k_1$ and the crossing lines $l_i$ and $\overline{p_0p_1}$.  Call these points $a_{0s}$ and $a_{1t}$.  Because $q \notin \overline{p_0p_1}$, they are different points and $s, t \neq x$.  Because $D \notin \overline{p_0p_1}$, $s \neq t$.  So  $\overline{p_0p_1} = l_{st}$. 

If $r \in l_0$ (or $l_1$, but they are similar), then (A\ref{Axiom:A6}) applied to the intersecting lines $k_0, k_1$ and the crossing lines $l_1$ and $\overline{p_0p_1}$ implies an intersection point $l_1 \cap \overline{p_0p_1}$, which is not $q$ so it is $a_{1t}$ for some $t \neq x$.  Applying the axiom again to intersecting lines $k_1$ and $l_1$ and crossing lines $l_0$ and $\overline{p_0p_1}$ implies an intersection point of the latter two lines.  This point cannot be $q$ because it is in $l_0$ but not $l_1$, so it is an $a_{0s}$ for some $s \neq x$.  We have found points $a_{0s}, a_{1t} \in \overline{p_0p_1}$, so that line is $l_{st}$.

Case \ref{ppgeneric}(b).  Assume $k_0 \in L$ and $k_1 = n_{xt'} \in N$.  The intersecting lines $k_0, k_1$ with crossing lines $l_0, \overline{p_0p_1}$ generate a crossing point $a_{0s} \in l_0 \cap \overline{p_0p_1}$ and with crossing lines $l_1, \overline{p_0p_1}$ they generate $a_{1t} \in l_1 \cap \overline{p_0p_1}$.  Thus, $\overline{p_0p_1} = l_{st}$.

Case \ref{ppgeneric}(c).  Last, assume $k_0, k_1 \in N$.  We need two steps.  Axiom (A\ref{Axiom:A6}) applied to intersecting $k_0, k_1$ and crossing $n_1, \overline{p_0p_1}$ gives a point $r \in n_1 \cap \overline{p_0p_1}$.  Now we use the intersecting lines $k_0, n_1$ and crossing lines $l_i, \overline{p_0p_1}$ for each of $l_0$ and $l_1$ to get points $a_{0s}, a_{1t} \in \overline{p_0p_1}$, proving that $\overline{p_0p_1} = l_{st}$.

Case \ref{ppl0}.  Let $p_0$ in $l_0 \setminus q$, say $p_0=a_{0s}$.  Since $p_0$ is in $l_0$, it cannot be in a line $n_{xt}$ (other than $l_0$ itself); thus, it is in a line $l_{st}$ for some $t \neq x,s$.  
If $p_1 \in x$ or $l_0$ or $s$ we are in a previous case.  If $p_1 \in l_1 \setminus q$, then $p_1 = a_{1t'}$ and $\overline{p_0p_1} = l_{st'}$ for some $t'$.  So, we assume $p_1$ is not in $l_1$.  If $p_1 \in l_{st}$ then $\overline{p_0p_1} = l_{st}$, so we assume $p_1 \notin l_{st}$.  We know it is in either an $l_{s't'}$ or an $n_{xt'}$.

\begin{figure}[ht]
(a)\ \includegraphics[width=6cm]{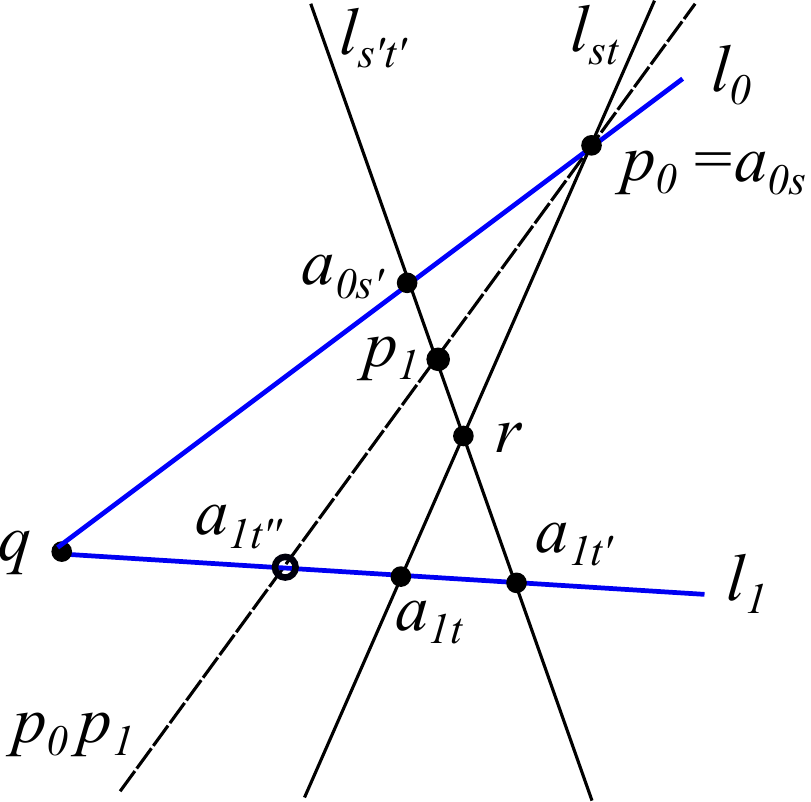}
\quad
(b)\ \includegraphics[width=6.5cm]{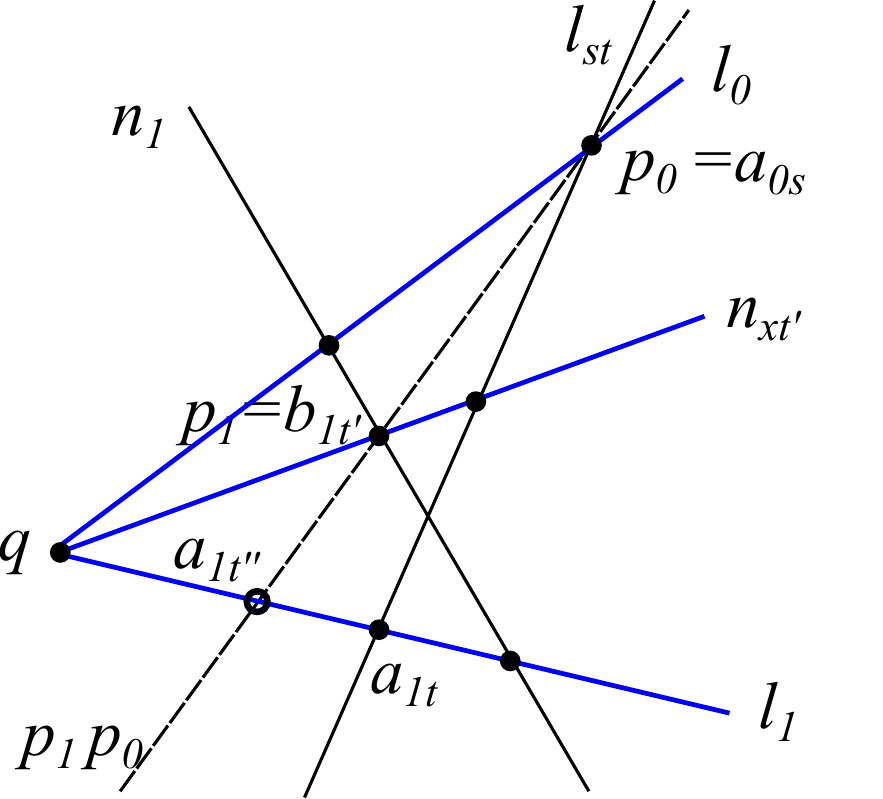}
\caption{For the proof of Cases \ref{ppl0}(a) and (b).}
\label{fig:twopointsgeneratingproof5}
\end{figure}

Case \ref{ppl0}(a).  Suppose $p_1 \in l_{s't'}$.  If $s=s'$, then $\overline{p_0p_1} = l_{s't'}$ and we are done, so we assume $s \neq s'$.  If $t \neq t'$, then $l_{st}, l_{s't'}$ form the intersecting lines (intersecting at a point $r \neq p_1$) and $l_1, \overline{p_0p_1}$ the crossing lines for (A\ref{Axiom:A6}); this implies an intersection point $a_{1t''} \in \overline{p_0p_1}$ and then $\overline{p_0p_1} = l_{st''}$.  We consider the case $t=t'$.  Then (A\ref{Axiom:A6}) applies to the intersecting pair $l_0, l_{s't'}$ and the crossing pair $l_1, \overline{p_0p_1}$, giving a point $a_{1,t''} \in l_1 \cap \overline{p_0p_1}$, so $\overline{p_0p_1} = l_{st''}$, and we are done.

Case \ref{ppl0}(b).  Suppose, however, that $p_1 \in n_{xt'}$ (but not in $x, l_0, l_1$).  Here we take as intersecting lines $l_0$ and $n_{xt'}$ with crossing lines $l_1$ and $\overline{p_0p_1}$.  This gives us a point $a_{1t''} \in l_1 \cap \overline{p_0p_1}$ so that $\overline{p_0p_1} = l_{st''} \in \cL_\PP$, provided that $p_0, p_1 \notin n_1$ so there are four distinct crossing points.  
For the exceptional cases, If both are in $n_1$ we are done: $\overline{p_0p_1} = n_1$.  
If $p_0 \in n_1$ and $p_1 \notin n_1$, then intersecting lines $n_1, n_{xt'}$ and crossing lines $l_1, \overline{p_0p_1}$ give us an intersection point $a_{1t''}$ so that $\overline{p_0p_1} = l_{tt''}$.  If the reverse, where $p_1 \in n_1$, so $p_1 = b_{1t'}$, and $p_0 \notin n_1$, the intersecting lines should be $l_0$ and $n_1$ and the crossing lines $l_1, \overline{p_0p_1}$, giving an intersection point $a_{1t''}$ and making $\overline{p_0p_1} = l_{st''}$---but with the one exception when $b_{1z} \in \overline{p_0p_1}$ since then the crossing lines make only three crossing points.  In that case $\overline{p_0p_1} = l_{sz} = l_{st}$, so we are done.

Case \ref{ppn1}.  We assume $p_0 \in n_1$ but not in any of $x,l_0,l_1$, and $p_1 \notin x, l_0, l_1, n_1$.  As with $p_1$ in Case \ref{ppl0}, $p_0$ is either in a line $l_{st}$ or a line $n_{xt}$.   Also as with Case \ref{ppl0}, there are exceptional subcases.

\begin{figure}[ht]
(a)\ \includegraphics[width=7cm]{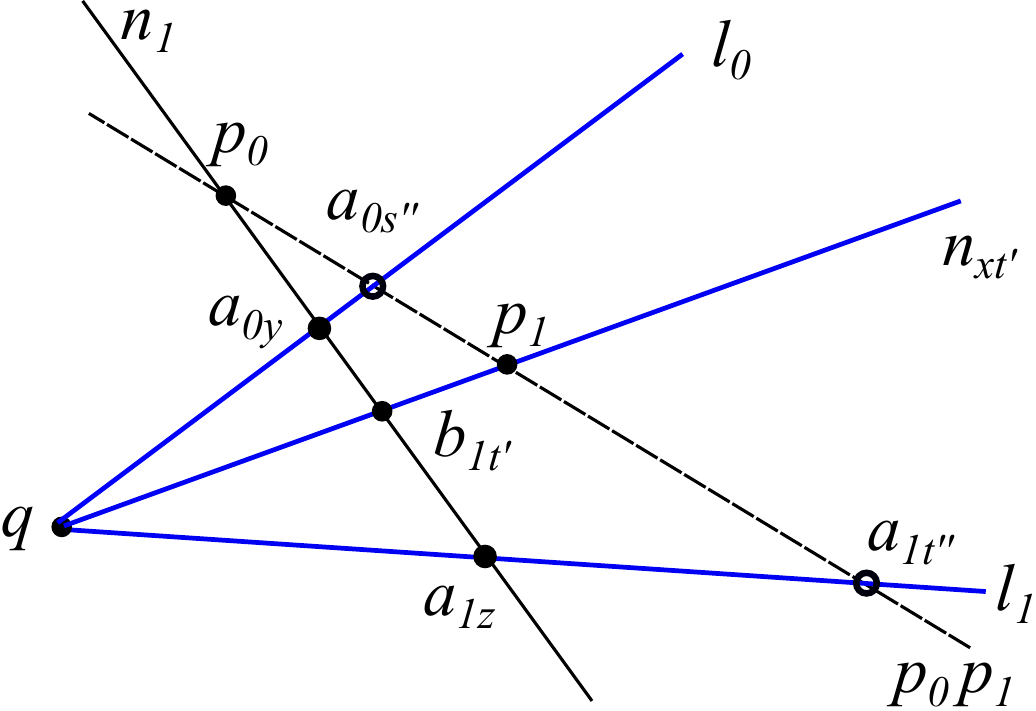}
\quad 
(b)\ \includegraphics[width=6cm]{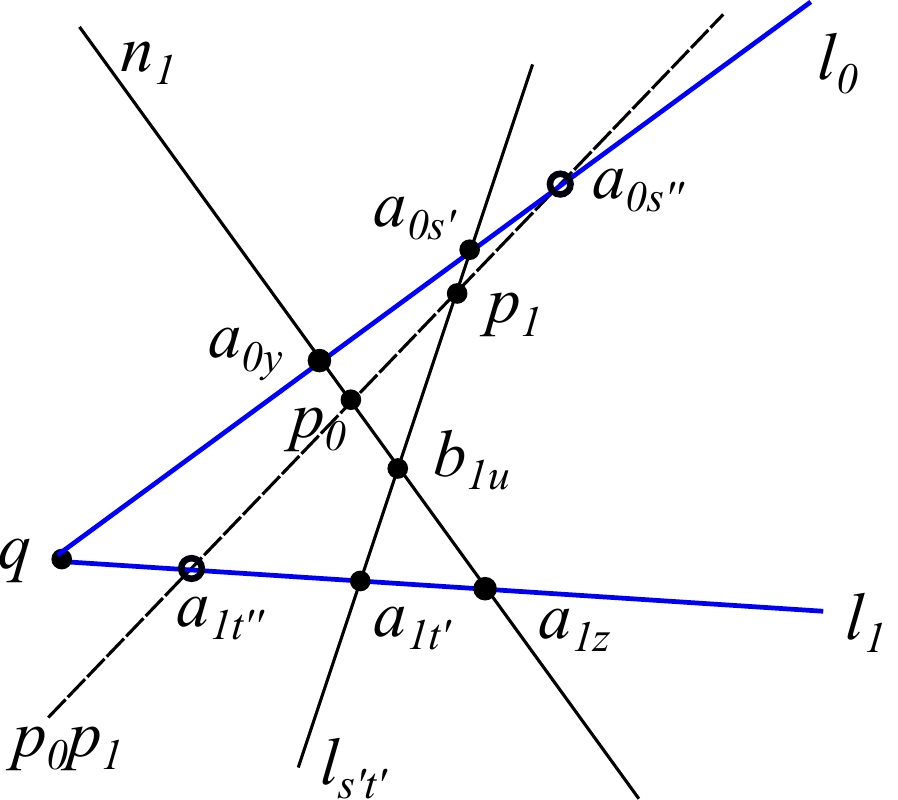}
\caption{For the proof of Cases \ref{ppn1}(a) and (b).}
\label{fig:twopointsgeneratingproof6}
\end{figure}

Case \ref{ppn1}a.  Suppose $p_1$ is in an $n_{xt'}$.  We assume $p_0 \notin n_{xt'}$, since otherwise $\overline{p_0p_1} = n_{xt'}$.  There are two steps.  First, if $p_0 \in l_0$, then $p_0 = a_{0y}$; let $s''=y$ and go to the second step.  Otherwise, apply (A\ref{Axiom:A6}) to intersecting lines $n_1, n_{st'}$ and crossing lines $l_0, \overline{p_0p_1}$ to get a point $a_{0s''} \in l_0 \cap \overline{p_0p_1}$.  In the second step, the intersecting lines are the same and the crossing lines are $l_1, \overline{p_0p_1}$, giving a point $a_{1t''}$ and a line $l_{s''t''} = \overline{p_0p_1}$.

Case \ref{ppn1}b.  Suppose $p_1 \in l_{s't'}$.  If $p_0 \in l_{s't'}$, that is the line $\overline{p_0p_1}$.  Otherwise we apply (A\ref{Axiom:A6}) simultaneously to two situations.  The intersecting lines are $n_1, l_{s't'}$ and the crossing lines are $l_i, \overline{p_0p_1}$ for $i=0,1$.  This gives points $a_{0s''}, a_{1t''} \in \overline{p_0p_1}$ so $\overline{p_0p_1} = l_{s''t''} \in \cL_\PP$, and we are done.

That concludes the cases.  As they are exhaustive, every pair of points of $\Pi$ is joined by a line of $\Pi$, which establishes the existence of a projective plane on $l_0$ and $l_1$.  The plane is unique because by Axiom (A\ref{Axiom:A1}) the line joining any two points in $\PR$ is unique so, given $l_0$ and $l_1$, there is no choice in constructing $\Pi$.
\end{proof}

An interpretation of Theorem \ref{prop:twolinesintersectingpp} is the following corollary.

\begin{cor} \label{cor:3pointssubplane}
Given three noncollinear ordinary points in a projective rectangle $\PR$, there is a unique full projective plane in $\PR$ that contains all three points.

Given an ordinary line $l$ and an ordinary point $p$ not in $l$, there is a unique full projective plane in $\PR$ that contains both.
\end{cor}

\begin{proof}
For the first part, let the three points be $p,q,r$.  No special line contains all three, so there is one, say $p$, that is not in a special line through the others.  The lines $pq$ and $pr$ are ordinary lines, they are distinct by noncollinearity of the three points, and they intersect, so by Theorem \ref{prop:twolinesintersectingpp} there is a unique full projective plane that contains them and the three points.

The second part follows by taking $q,r \in l$.
\end{proof}

\begin{thm}\label{prop:maximalsubplane}
In a projective rectangle, every maximal subplane is full.
\end{thm}

\begin{proof}
The line set of an incidence subplane $\pi$ contains two ordinary lines $l_1,l_2$ and its point set contains their intersection point.  It follows from Theorem \ref{prop:twolinesintersectingpp} that $\pi$ is a subplane of the full subplane determined by $l_1$ and $l_2$.  
\end{proof}

Thus, maximality and fullness are equivalent for projective subplanes of a projective rectangle.  
From now on, when we refer to a \emph{plane} in a projective rectangle, we mean a full projective subplane.  Also, when we say several lines are \emph{coplanar}, we mean there is a plane $\pi$ such that each of the lines that is ordinary is a line of $\pi$ and for each line $s$ that is special, $s \cap \pi$ is a line of $\pi$.

We can now characterize a nontrivial projective rectangle as a projective rectangle that contains more than one maximal projective subplane.  Such projective rectangles have properties not common to all projective planes; e.g., they satisfy the dual half of Desargues's Theorem  (see Theorem \ref{thm:axialdesargues}) and they are harmonic matroids (see \cite{pr3}).

\begin{cor}\label{coro;lineinaprojectiveplane}
Let $\PR$ be a projective rectangle. Every ordinary line in $\PR$ is a line of a plane in $\PR$.  If\/ $\PR$ is nontrivial, then every ordinary line $l$ is a line of at least three planes that contain $l$.
\end{cor}

\begin{proof} 
Let $l$ be an ordinary line in $\PR$. From Theorem  \ref{numberofpoinsinlines} Part (\ref{numberofpoinsinlines:b}) we know that there
is another ordinary line $l'$ that intersects $l$ at exactly one point. This and Theorem \ref{prop:twolinesintersectingpp}
imply that $l$ is in a plane $\pi$.

If\/ $\PR$ is nontrivial, there is a point $q$ not in $\pi$.  Let $p_1,p_2 \in \pi$ be points in $l$ that are not in the special line that contains $q$.  Then the plane $\overline{p_1p_2q}$ that contains both ordinary lines $\overline{p_1q}$ and $\overline{p_2q}$, which exists and is unique by Theorem \ref{prop:twolinesintersectingpp}, is a plane containing $l$ that is different from $\pi$.  To find a third plane, let $p_1 \in \pi_1$ and $p_2 \in \pi_2$ be ordinary points not in $l$.  There is an ordinary line $\overline{p_1p_2}$ that must contain a third point $p_3$ since $m\geq2$ by Theorem \ref{numberofpoinsinlines}.  By Corollary \ref{cor:3pointssubplane} there is a unique plane $\pi_3$ that contains $l$ and $p_3$.
\end{proof}

\begin{cor}\label{linesinplanes}
If $s$ is a special line in the projective rectangle $\PR$ and $\pi$ is a plane in $\PR$, then $s \cap \pi$ is a line of $\pi$.
\end{cor}

\begin{proof}
Let $p_1$ and $p_2$ be points in distinct special lines that are not $s$.  Then by Axiom (A\ref{Axiom:A6}) there is an ordinary line $l$ that contains both $p_1$ and $p_2$, and by Corollary \ref{coro;lineinaprojectiveplane} there is a plane $\pi$ that contains $l$.  
In $\pi$ there is another line $l'$ that intersects $l$ at $p_1$; then $q=l\cap s$ and $q'=l' \cap s$ are two points in $s \cap \pi$, which determine a line in $\pi$ that is contained in the unique line $s$ of\/ $\PR$ that contains $q$ and $q'$.  Thus, $s \cap \pi$ is a line of $\pi$.
\end{proof}

Now we prove a generalization of Theorem \ref{prop:twolinesintersectingpp} to all lines, although we lose uniqueness of the containing plane.

\begin{cor}\label{cor:coplanaralllines}
Let $\PR$ be a projective rectangle.  If two lines $l_1$ and $l_2$ intersect in a point $p$, then they are coplanar.
\end{cor}

\begin{proof}
Suppose $l_1$ is a special line.  There are points $p_1$ in $l_1 \setminus l_2 \setminus D$ and $p_2$ in $l_2 \setminus l_1$.  By Axiom (A\ref{Axiom:A6}) there is an ordinary line $l_3$ determined by $p_1$ and $p_2$.  

If $l_2$ is ordinary, by Theorem \ref{prop:twolinesintersectingpp} there is a unique plane $\pi$ that contains $l_2$ and $l_3$.  By Proposition \ref{linesinplanes} the restriction of $l_1$ to $\pi$ is a line of $\pi$, so $l_1$ and $l_2$ are coplanar.  

If $l_2$ is special, then $l_3$ is ordinary.  By Proposition \ref{coro;lineinaprojectiveplane} there is a plane $\pi$ that contains $l_3$, and by Proposition \ref{linesinplanes} both $l_1\cap\pi$ and $l_2\cap\pi$ are lines of $\pi$.  Thus, $l_1$ and $l_2$ are coplanar.
\end{proof}

Next is an intersection property of lines that has a consequence for the matroid structure of a projective rectangle.

\begin{prop}\label{threelinesintersect}
Suppose three lines in a projective rectangle $\PR$ intersect pairwise in three different points.  Then they are a coplanar triple.

Equivalently, if three lines intersect pairwise (i.e., are pairwise coplanar) but are not a coplanar triple, then they all intersect in the same point.
\end{prop}

\begin{proof}
Suppose two ordinary lines $l_1, l_2$ intersect in a point $p$ and lie in a common plane $\pi$, and suppose a third line $l_3$, possibly special, intersects $l_1$ and $l_2$ in points different from  $p$.  Choosing any points $q_1 \in l_1 \setminus p$ and $q_2 \in l_2 \setminus p$ determines a line of $\pi$ through $q_1$ and $q_2$.  By Construction \ref{Definition:Projective:Plane} and Theorem \ref{prop:maximalsubplane}, this line is either an ordinary line of\/ $\PR$ or the restriction to $\pi$ of a special line of\/ $\PR$.  In particular, this applies to $l_3$, hence $l_1$, $l_2$ and $l_3$ are a coplanar triple of lines of\/ $\PR$.

In case $l_1$ is ordinary while $l_2$ and $l_3$ are special, by Corollary \ref{cor:coplanaralllines} $l_1$ and $l_2$ are coplanar in a plane $\pi$ and by Proposition \ref{prop:Dinpp} $l_3\cap\pi$ is a line of $\pi$, so the three lines are coplanar.

The second statement, which is the contrapositive of the first (and see Corollary \ref{cor:coplanaralllines}), is a useful restatement.
\end{proof}

\begin{prop}\label{prop:linesamenumberpts}
If a finite projective rectangle has order $(n,n)$, then it is a projective plane.
\end{prop}

\begin{proof}
Because $n=m$, the projective plane of Corollary \ref{coro;lineinaprojectiveplane} is the whole projective rectangle.
\end{proof}

This proposition does not apply to the infinite case; see Example \ref{ex:infinitesubplane}.

\subsection{No Vamos configuration}\label{sec:vamos}\

\begin{figure}[htb]
\includegraphics[scale=.25]{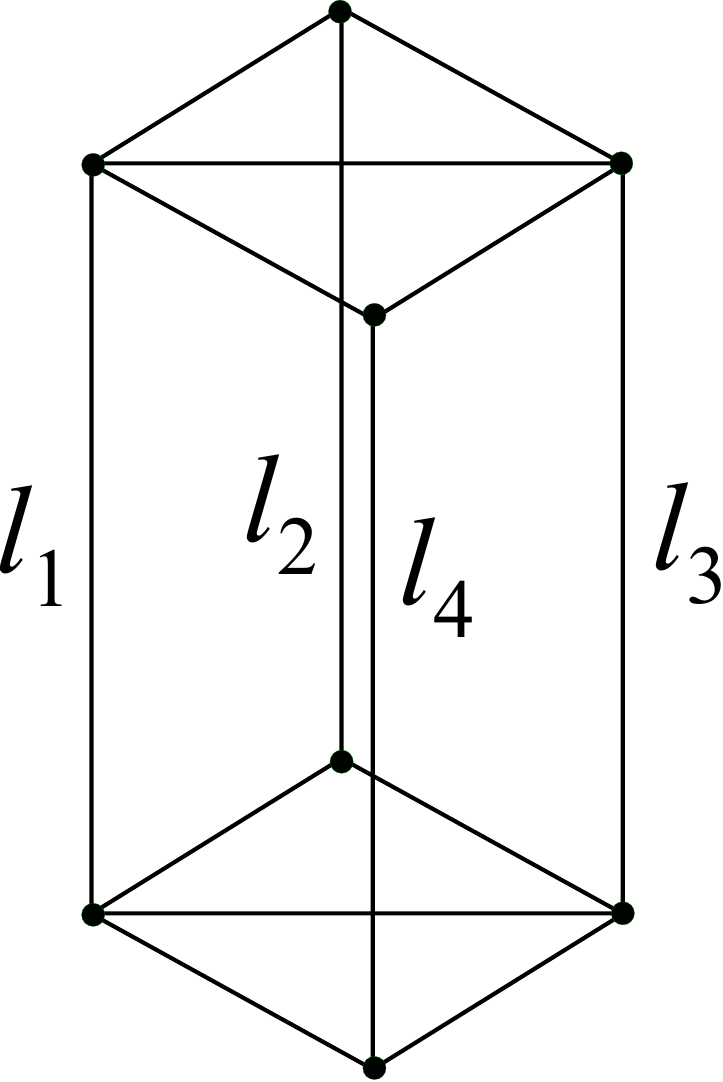}
\caption{The Vamos matroid and incidence structure.  The vertical lines are coplanar in pairs $l_1l_2$, $l_2l_3$, $l_1l_3$, $l_1l_4$, and $l_3l_4$, but not $l_2l_4$.  No three of them are coplanar.  The top and bottom squares may be planes or not;  that does not affect Corollary \ref{novamos}.}
\label{F:vamos}
\end{figure}

The Vamos matroid is the matroid of eight points in Figure \ref{F:vamos}.  It is one of the smallest matroids that cannot be represented in a projective geometry; for that reason it is one of the fundamental matroid examples.  
However, we shall not think of it as a matroid but as an incidence structure with eight points as well as lines and planes.  The lines are the solid lines in Figure \ref{F:vamos} and the planes are the ones composed of pairs of lines as described in the caption.  
(As a matroid a projective rectangle has rank 3 while the Vamos matroid has rank 4 and therefore it is trivial that it cannot be a submatroid of a projective rectangle.  That is why it is important to think of the Vamos incidence structure instead of the Vamos matroid, even though they look the same in a diagram.)

\begin{cor}\label{novamos}
The Vamos incidence structure is not a substructure of any projective rectangle.
\end{cor}

\begin{proof}
Suppose a configuration of this kind exists in a projective rectangle.  By Proposition \ref{threelinesintersect} the lines $l_1,l_2,l_3$ are concurrent in a point and the lines $l_2,l_3,l_4$ are also concurrent in a point.  Clearly, these points are one point, so $l_1$ and $l_3$ contain a common point and hence are coplanar, contrary to the structure of the Vamos matroid.  That proves the corollary.
\end{proof}

\sectionpage\section{Finite projective rectangles}\label{sec:finite}

In finite projective rectangles there are many possibilities for counting elements and configurations.  They are the topic of this section. 

\subsection{Counts}\label{sec:counting}\

We extend the counts of points, lines, etc.\ in Section \ref{sec:fundamental} to planes and various kinds of incidence.

\begin{thm}\label{thm:counting}
Let $\PR$ be a finite projective rectangle of order $(m,n)$.  
\begin{enumerate}[{\rm(a)}]

\item \label{prop:counting:lmeetsl} The number of ordinary lines that are concurrent with each ordinary line is $(m+1)(n-1)$.

\item \label{prop:counting:plonpi} There are $m(m+1)$ ordinary points and $m^2$ ordinary lines in each plane. 

\item \label{prop:counting:pl} The number of pairs $(p,l)$ that consist of an ordinary point $p$ and an ordinary line $l$ that contains $p$ is $(m+1)n^2$.

\item \label{prop:counting:pionl} The number of planes that contain each ordinary line is $\dfrac{n-1}{m-1}.$

\item \label{prop:counting:lpi} The number of pairs $(l,\pi)$ such that $l$ is an ordinary line and $\pi$ is a plane that contains $l$ is $n^2 \dfrac{n-1}{m-1}.$

\item \label{prop:counting:pi} The number of planes in $\PR$ is $\dfrac{n^2(n-1)}{m^2(m-1)}.$

\item \label{prop:counting:lpionp} For a fixed ordinary point $p$, the number of triples $(p,l,\pi)$ such that $l$  is an ordinary line incident with $p$ and $\pi$ is a plane that contains $l$ is $n \dfrac{n-1}{m-1}.$

\item \label{prop:counting:plpi} The number of triples $(p,l,\pi)$ such that $p$ is an ordinary point, $l$ is an ordinary line, and $\pi$ is a plane that contains $l$ is $(m+1)n^2 \dfrac{n-1}{m-1}.$

\item \label{prop:counting:ppi} The number of pairs $(p,\pi)$ such that $p$ is an ordinary point and $\pi$ is a plane that is incident with $p$ is $\dfrac{(m+1)n^2}{m}\, \dfrac{n-1}{m-1}.$

\item \label{prop:counting:pionp} The number of planes that are incident with each ordinary point is $\dfrac{n}{m}\, \dfrac{n-1}{m-1}.$

\end{enumerate}
\end{thm}

\begin{proof}
Proof of \eqref{prop:counting:lmeetsl}.  Let $l$ be an ordinary line.  From Part \eqref{numberofpoints}) there are $m+1$ points on $l$.  From Theorem \ref{numberofpoinsinlines} Part (\ref{cor:numberofpoinsinlines:b}) we know there are $n-1$ ordinary lines that intersect $l$ at each point.  All those lines are distinct.

Proof of \eqref{prop:counting:plonpi}.  This follows from the fact that the plane is projective of order $m$.  We exclude the one special point $D$ and the $m+1$ special lines in the plane.

Proof of \eqref{prop:counting:pl}.  Each of the $n^2$ ordinary lines (Theorem \ref{numberofpoinsinlines} Part \eqref{numberofpoinsinlines:h}) contains $m+1$ ordinary points (Part \eqref{numberofpoinsinlines:d}).

Proof of \eqref{prop:counting:pionl}.  
Let $l$ be an ordinary line. From Part \eqref{prop:counting:lmeetsl} there are $(m+1)(n-1)$ ordinary lines $l'$ that intersect $l$ at exactly one point.
Theorem \ref{prop:twolinesintersectingpp} guarantees the existence of a unique plane $\pi$ that contains
both $l$ and $l'$.  By Part \eqref{prop:counting:plonpi} the number of ordinary lines in $\pi$ that intersect $l$ is $m^2-1 = (m+1)(m-1)$.  Thus, the number of planes on $l$ is the quotient, $(m+1)(n-1)/(m+1)(m-1)=(n-1)/(m-1)$.

Proof of \eqref{prop:counting:lpi}.  The number of ordinary lines should be multiplied by the number of planes on each line.

Proof of \eqref{prop:counting:pi}.  The number of incident line-plane pairs should be divided by the number of ordinary lines in a plane.

Proof of \eqref{prop:counting:lpionp}.  The number of incident line-plane pairs should be multiplied by the number of points in an ordinary line.

Proof of \eqref{prop:counting:plpi}.  The number of triples in Part \eqref{prop:counting:lpionp} should be multiplied by the number of ordinary points from Part \eqref{numberofpoints}.

Proof of \eqref{prop:counting:ppi}.  The number of triples in Part \eqref{prop:counting:plpi} should be divided by the number of ordinary lines in $pi$ that contain $p$, which is $m$.

Proof of \eqref{prop:counting:pionp}.  Either divide the number of triples in Part \eqref{prop:counting:plpi} by $m$, the number of ordinary lines on $p$ in $\pi$, or divide the number in Part \eqref{prop:counting:ppi} by $(m+1)n$, the whole number of ordinary lines on $p$.
\end{proof}

Two lines are \emph{skew} if they have no point in common.
A \emph{skew class} of lines is a maximal set of lines, in which every pair is skew.  If a line has no skew mate, it is a skew class of one.  A line may belong to more than one skew class.  Two lines that are skew to the same line may intersect.

\begin{thm} \label{numberofpoinsinlines:finite} 
If\/ $\PR$ is a projective rectangle of finite order $(m,n)$, then the following hold in $\PR$:
\begin{enumerate}[{\rm(a)}]

\item  \label{LinesPParalletL} Given an ordinary point $p$ and given any ordinary line
$l$ that does not contain $p$, there are exactly $n-m$ ordinary lines containing $p$ that are skew to $l$.

\item\label{cor:specialnonspecialskew} If $l$ is an ordinary line, then there are $(n-1)(n-m)$ lines that are skew to $l$.

\item \label{skewconcurrent}  If $l_{1}$ is skew to $l$, there are $(m+1)(n-m)$ lines skew to $l$ that are concurrent with $l_1$.

\end{enumerate}
\end{thm}

\begin{proof}
Proof of Part \eqref{LinesPParalletL}. From Theorem \ref{numberofpoinsinlines} Part \eqref{numberofpoinsinlines:e}
we know that there are exactly $n+1$ lines passing through $p$ (including a special line).
From Theorem \ref{numberofpoinsinlines} Part \eqref{numberofpoinsinlines:b} we also know that there are exactly $m+1$ lines
passing through $p$ that intersect $l$ (including a special line). Therefore, there are exactly $n-m$
ordinary lines passing through $p$ and skew to $l$.

Part \eqref{cor:specialnonspecialskew} follows by subtracting from the number of ordinary lines, $n^2$ (Theorem \ref{numberofpoinsinlines} Part \eqref{numberofpoinsinlines:h}), the number that are concurrent with $l$, which is $(m+1)(n-1)$ (Theorem \ref{thm:counting} Part \eqref{prop:counting:lmeetsl}), and the number that are $l$, which is $1$.

Part \eqref{skewconcurrent} follows from Part \eqref{LinesPParalletL}.
\end{proof}

\begin{prop} \label{Parallel:Class} Suppose that $\PR$ is a nontrivial projective rectangle of finite order $(m,n)$.  
Let $l$ be an ordinary line in $\PR$. 
Tthere is a skew line class containing $l$ that has at least $m+1$ lines in it. 

Let $M = \lceil (n-m)/m \rceil - m+1$, the largest integer such that $n/m>m+1+M$. Then there is a skew class containing $l$ that has at least $m+1+M$ lines in it. 
\end{prop}

\begin{proof}
Let $l$ be an ordinary line and let $l_1\ne l$ be an ordinary line passing though $q\in l$. Let $p \ne q$ be a second point in $l$.
By Theorem \ref{numberofpoinsinlines:finite} Part \eqref{LinesPParalletL}, since $n>m$ there is an ordinary line $l_2$ passing through
$p$ skew to $l_1$. Let $a_i$ and $b_i'$ be the points in $l_1$ and $l_2$ for $i=1,2, \dots, m+1$, labeled so that the line $a_ib_i'$ is special.  Lines $a_ib_i$ and $a_jb_j$ for $i,j\in \{1,2, \dots, m+1\}$ with $i\ne j$, $b_i \neq b_j$, $b_i \neq b_i'$, and $b_j \neq b_j'$ are ordinary and are skew to each other, because 
if they intersect, then by Axiom (A\ref{Axiom:A6}), $l_1$ intersects $l_2$, which is a contradiction.
Note that it is easy to choose all $b_i \neq b_i'$ since $m\geq1$.  
Also, we can suppose that $l$ is the line $a_1b_1$.

Now we suppose that $n/m-m-1>0$ and $M$ is the largest integer such that  $n/m>m+1+M$.  (Thus, $n>m+M$.)  Let $s$ be a special line with points 
$s_1, s_2, \dots, s_m, \dots, s_{n},D$. Suppose that $s\cap a_ib_i=s_i$  for $i=1, \dots, m+1$. We prove by induction that there are lines 
$h_1,  h_2, \dots,  h_{M}$, skew to one other and to all lines of the form $a_ib_i$. 
Assume we have $k$ lines  $h_1,  h_2, \dots,  h_{k}$ that are skew to one other and to 
all lines of the form  $a_ib_i$ for some $k\in \{0,1, \dots, M-1\}$, where $s_{m+1+t} \in h_t$ for $t=1, 2, \dots, k$. 
First note that neither $h_{t}$ nor $a_ib_i$ contains the point $s_{m+1+k+1}$ and that $m(m+1+k)$ is the number of points in 
$(\bigcup_{t=1}^{k} h_{t} \cup \bigcup_{i=1}^{m+1} a_ib_i)\setminus s$. 
Thus, the maximum number of ordinary lines passing through $s_{m+1+k+1}$ intersecting a line of the form $a_ib_i$ and the lines 
$h_1, \dots, h_{k}$  is $m(m+1+k)$. Since $s_{m+1+k+1}$ is an ordinary  point, by Theorem \ref{numberofpoinsinlines} Part \eqref{numberofpoinsinlines} 
we know there are $n$ ordinary lines passing through this point.  Since $n>m(m+1+k)$ there must be at least one ordinary line $h_{k+1}$ 
passing through $s_{m+1+k+1}$ that is skew to all lines of the form $a_ib_i$ and the lines  $h_1, \dots, h_{k}$. This proves the induction, completing the proof. 
\end{proof}

In the notation of Theorem \ref{thm:constraint}, $M = (\tau-1)(m+1) - 2\tau$.  This is negative or zero if $\tau = 1$, or if $\tau=2$ and $m\leq 3$, and positive otherwise, so in the ``otherwise'' case the second bound on the maximum size of the skew class is the better one.

\subsection{Constraints on the parameters}\label{sec:constraint}\

We have found some integers in Theorem \ref{thm:counting}, namely, 
$$
\rho=\dfrac{n-1}{m-1}, \quad \dfrac{n}{m}\, \dfrac{n-1}{m-1}, \text{ and } \dfrac{n^2}{m^2}\, \dfrac{n-1}{m-1}.
$$
These integral fractions imply relationships between $m$ and $n$.
Theorem \ref{thm:constraint} is a constraint on $n$, given a value of $m$.  By Section \ref{sec:planes} $m$ must be the order of a projective plane; that is the only constraint we know on $m$.  Recall that $\cP$ is the point set of the projective rectangle.

\begin{prop} \label{prop:planesonspecials}
Let $p,p'$ be two ordinary points in a special line $s$ of a finite or infinite projective rectangle.  The planes $\pi$ that contain both $p$ and $p'$ partition $\cP\setminus s$ into sets $\pi\setminus s$ of size $m^2$.  For each other special line $s' \neq s$, the same planes partition $s'\setminus D$ into sets $\pi\cap(s'\setminus D)$ of size $m$, and each such set is in a unique plane that contains $p$ and $p'$.  When $m, n$ are finite there are $n/m$ such planes.
\end{prop}

\begin{proof}
For an ordinary point $q\in s'$ let $\pi(q)$ denote the plane that contains $p,p',q$.  This plane is unique, by Theorem \ref{prop:twolinesintersectingpp}, because it is determined by the intersecting ordinary lines $pq$ and $p'q$.  
Choose another ordinary point $q' \in s' \setminus \pi(q)$ and suppose $\pi(q)$ and $\pi(q')$ contain a common point $r \notin s$.  Then both planes contain the intersecting ordinary lines $pr$ and $p'r$, so they must be the same plane.  It follows that the distinct planes $\pi(q)$ for $q \in s' \setminus D$ partition the points not in $s$ and in particular those of $s' \setminus D$.  The intersection $\pi(q) \cap s'$ is a line of $\pi(q)$ that contains $D$, so the number of ordinary points in it is $m$ and the total number of points in $\pi(q) \setminus s$ is $m^2$.  In the finite case the number of sets into which $s' \setminus D$ is partitioned is therefore equal to $n/m$, and this is the number of planes that contain both $p$ and $p'$.
\end{proof}

\begin{thm}\label{thm:constraint}
For a projective rectangle $\PR$ of finite order $(m,n)$, there is an integer $\tau \geq 0$ such that $n = m + \tau m(m-1)$.  If $\PR$ is nontrivial, then $\tau \geq 1$.
\end{thm}

\begin{proof}
Integrality of $(n-1)/(m-1)$ implies that there is an integer $\rho \geq 1$ such that $n = 1 + \rho(m-1)$.  
Proposition \ref{prop:planesonspecials} implies that $n = \sigma m$ for some positive integer $\sigma$.  Therefore, $n = \rho(m-1)+1 = \sigma m$.  It follows that $(\rho-\sigma)m = \rho-1$, so $\rho-1$ is a multiple of $m$, say $\rho = \tau m+1$ where $\tau\geq0$.  Then substituting for $\rho$ gives $(\tau m+1-\tau)m = \sigma m$, and upon division by $m$ we find that $\sigma = \tau(m-1) + 1$.  This implies $n = \tau m(m-1) + m$, so $n-m = n-m = \tau m(m-1)$.
\end{proof}

We infer the expressions 
\begin{equation}
\label{eq:tauexpressions}
\begin{gathered}
\dfrac{n-1}{m-1} = \tau m+1, \quad \dfrac{n}{m} = \tau(m-1)+1, \\
\dfrac{n}{m}\,\dfrac{n-1}{m-1} = [\tau(m-1)+1] [\tau m+1], \\ 
\dfrac{n^2}{m^2}\, \dfrac{n-1}{m-1} = [\tau(m-1)+1]^2 [\tau m+1].
\end{gathered}
\end{equation}

\begin{cor}\label{cor:constraint-lbn}
If the projective rectangle is nontrivial, $n \geq m^2$ and $\rho \geq m+1$.
\end{cor}

\begin{exam}\label{ex:constraintm=3}
If the projective rectangle has $m=2$, then $n= 2 + 2\tau$, where $\tau\geq0$.  The value $\tau=0$ gives the Fano plane and $\tau=1$ gives $n=4$ as with the $L_2^2$ projective rectangle of Example \ref{ex:L2k}.  
However, not all those values of $\tau$ admit a projective rectangle with $m=2$; there are examples only for $n = 2^k$, that is, for $\tau = 2^{k-1}-1$ (see Section \ref{sec:narrow}).  Our numerical constraints need strengthening.
\end{exam}

\sectionpage\section{Axial and central Desargues's theorems}\label{sec:desargues}

\newcommand\DD{C}
\newcommand\dd{c}

Consider two triangles in a projective rectangle, $A = \triangle a_1a_2a_3$ and $B = \triangle b_1b_2b_3$.  (A triangle consists of three points, not all collinear, and the three lines joining the points in pairs.)  We always assume the six vertices $a_i$ and $b_i$ are distinct.  There are three lines $l_i = \overline{a_ib_i}$; if they concur in a point $p$ we say the triangles are \emph{centrally perspective from center $p$}. 
(The three lines are determined by the subscript correspondence of the two triangles' vertices.)  
If each of the three pairs of lines $\overline{a_ia_j}$ and $\overline{b_ib_j}$ meets in a point $p_{ij}$ and the points $p_{12}, p_{13}, p_{23}$ are collinear in a line $l$, we say $A$ and $B$ are \emph{axially perspective from axis $l$}.  
The Central Desargues's Theorem says that, if two triangles are centrally perspective, then they are axially perspective.  The converse is the Axial Desargues's theorem.  The two together are generally known as Desargues's Theorem.

In a projective plane the points $p_{ij}$ always exist.  However, neither half of Desargues's Theorem is valid in every projective plane; in fact the validity of Desargues's Theorem is equivalent to the existence of plane coordinates in a division ring.  Thus, for any plane, knowing whether Desargues's theorem holds true is a fundamental question.

Every projective plane is a projective rectangle, so we cannot say that Desargues's Theorem holds true in every projective rectangle; but eliminating projective planes from consideration changes the situation.  We first establish that each triangle in the axial configuration is contained in a plane.

\begin{lem}\label{lem:axialplane}
If $A=\triangle a_1a_2a_3$ is a triangle and $l$ is a line that intersects the three lines $\overline{a_ia_j}$ in three points $p_{ij}$, then all six points and the four lines are contained in a unique plane. 
\end{lem}

\begin{proof}
There are four lines in the configuration of six points:  $l$ and the lines $l_{ij} = \overline{a_ia_j}$.  At most two can be special, so two are ordinary, say $l'$ and $l''$.  Any two of the four lines intersect, so $l'$ and $l''$ intersect; this implies they are in a unique plane $\pi$ (by Theorem \ref{prop:twolinesintersectingpp}).  The other two lines of the four are each determined by one point in $l$ and one in $l'$, so each is a line of $\pi$, or if special the intersection with $\pi$ is a line of $\pi$.  
\end{proof}

First we prove that a nontrivial projective rectangle $\PR$ satisfies the Axial Desargues's Theorem when the axis is an ordinary line.  We do not know whether the assumption that the axis is ordinary can be removed.

\begin{thm}[Ordinary Axial Desargues's Theorem]\label{thm:axialdesargues}
In a nontrivial projective rectangle $\PR$, if two triangles are axially perspective and the axis is an ordinary line, then they are centrally perspective.  
\end{thm}

\begin{figure}[htb]
\includegraphics[width=8cm]{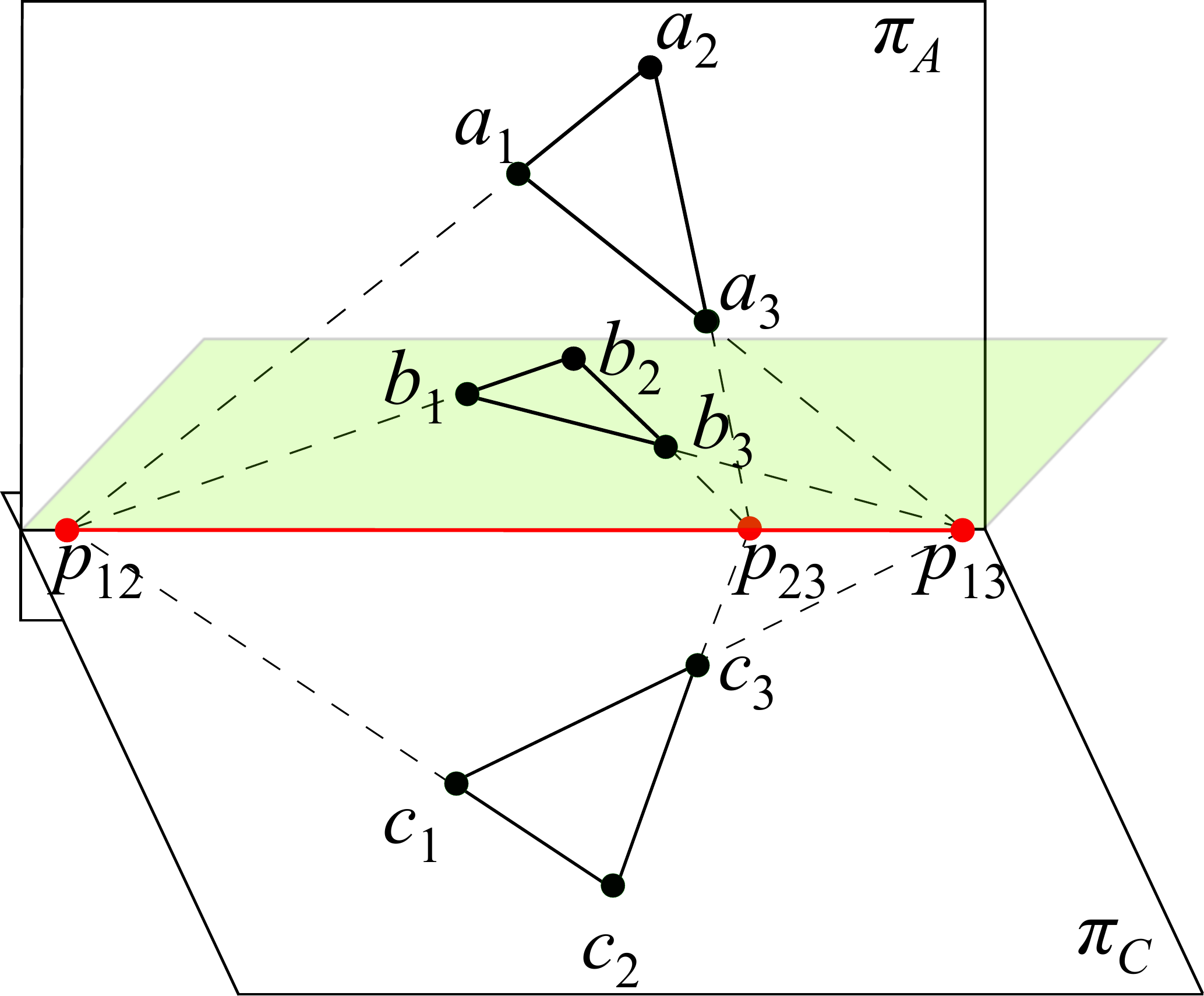}
\caption{Triangles $A$ and $B$ that are axially perspective, with the extra triangle $\DD$.  $A$ and $B$ may or may not be in the same plane.}
\label{F:axialperspective}
\end{figure}

\begin{proof}
We begin by assuming triangles $A = \triangle a_1a_2a_3$ and $B = \triangle b_1b_2b_3$ are in planes $\pi_A$ and $\pi_B$, respectively, and are axially perspective from an ordinary line $l$ with intersection points $p_{ij}$, as in Figure \ref{F:axialperspective}.  The two planes may be the same or different; if they are different, $l$ is their intersection.  

If $\overline{a_1b_1}, \overline{a_2b_2}, \overline{a_3b_3}$ are not all coplanar, they are coplanar in pairs, since $a_i,b_i,a_j,b_j \in \overline{p_{ij}a_ia_j}$.  Hence, by Proposition \ref{threelinesintersect} there is a point $q$ at which all three lines are concurrent; therefore, $q$ is a center of perspectivity for $A$ and $B$.  
Thus, we assume henceforth that $\overline{a_1b_1}, \overline{a_2b_2}, \overline{a_3b_3}$ are all in one plane, so that $\pi_A = \pi_B$.

There is another plane $\pi_\DD$ on $l$ because $\PR$ is nontrivial and $l$ is ordinary (by Corollary \ref{coro;lineinaprojectiveplane}), and in this plane we can find a triangle $\DD = \triangle \dd_1\dd_2\dd_3$ that is axially perspective from $l$ with the same intersection points $p_{ij} = l \cap \overline{\dd_i\dd_j}$.

The lines $\overline{b_i\dd_i}$ and $\overline{b_j\dd_j}$ are coplanar in a plane $\overline{p_{ij}b_i\dd_j} = \overline{b_i\dd_ib_j\dd_j}$.  Therefore, they intersect in a point $s_{ij}$.  
The pairwise coplanar lines $\overline{b_1\dd_1}$, $\overline{b_2\dd_2}$, and $\overline{b_3\dd_3}$ are not all coplanar because $\overline{\dd_1\dd_2\dd_3} = \pi_\DD \not\ni b_1,b_2,b_3$.  By Proposition \ref{threelinesintersect}, those three lines have a common point $s = s_{12} = s_{13} = s_{23}$.  See Figure \ref{F:randsperspective}.

Similarly, there is a point $r = \overline{a_1\dd_1} \cap \overline{a_2\dd_2} \cap \overline{a_3\dd_3}$.

\begin{figure}[htb]
\includegraphics[width=8cm]{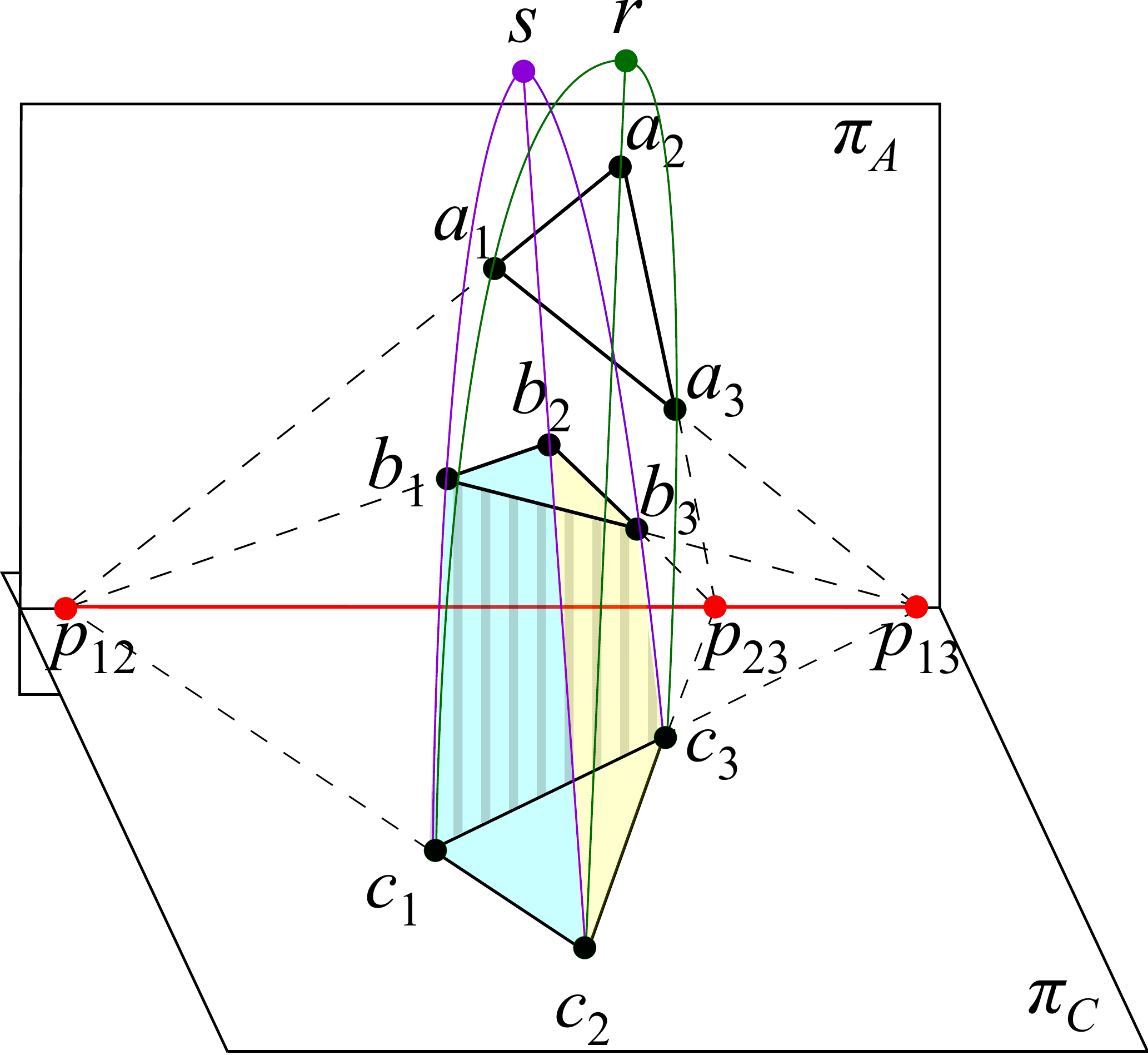}
\caption{The three planes $\overline{a_ib_i\dd_i}$ and the points $r$ and $s$.}
\label{F:randsperspective}
\end{figure}

We prove that $r \neq s$ and $r,s \notin \pi_A$.  If $r=s$, then $\overline{a_i\dd_i} = \overline{ra_i\dd_i} = \overline{r\dd_i}$ and $\overline{b_i\dd_i} = \overline{sb_i\dd_i} = \overline{r\dd_i}$, so $\overline{ra_i\dd_i}$ and $\overline{rb_i\dd_i}$ are the same line; that is, $a_i,b_i,\dd_i$ are collinear; but this is impossible.  
If $r$ or $s \in \pi_A$ then $c_1, c_2, c_3 \in \pi_A$, which contradicts $\pi_C \neq \pi_A$.

Each plane $\overline{a_ib_i\dd_i}$ contains $r$ and $s$ so the lines $\overline{a_ib_i}$ and $\overline{rs}$ are coplanar.  We know that $r,s \notin \overline{a_ib_i} \subset \pi_A$.  Hence, we have three triples $\overline{a_ib_i}, \overline{a_jb_j}, \overline{rs}$ of lines that are coplanar in pairs but not all coplanar.  By Proposition \ref{threelinesintersect} there is a point $q_{ij}$ at which each triple is concurrent.  Then taking $i=1$ and $j=2,3$, we have $q_{12} = \overline{rs} \cap \overline{a_1b_1} = q_{13}$, so $q_{12}=q_{13}$ is a point on all three lines $\overline{a_1b_1}, \overline{a_2b_2}, \overline{a_3b_3}$ and a center of perspectivity for $A$ and $B$.

That completes the proof.
\end{proof}

The case in which $A$ and $B$ are not coplanar is reminiscent of the higher-dimensional Desargues's Theorem for projective geometries.  That suggests a central Desargues's Theorem for noncoplanar triangles.  We are able to prove a mildly limited theorem.

\begin{thm}[Ordinary Higher Central Desargues's Theorem]\label{thm:highercentraldesargues}
Let $\PR$ be a nontrivial projective rectangle.  If two triangles in different planes are centrally perspective by ordinary lines, then they are axially perspective.
\end{thm}

\begin{proof}
We begin by assuming triangles $A$ and $B$ are in two different planes, $\pi_A$ and $\pi_B$ respectively, and are centrally perspective from a point $p$, such that all the lines $\overline{pa_ib_i}$ of perspectivity are ordinary.

OLD\\
Let $l_i := \overline{a_ib_i}$ (which exists and contains $p$ by central perspectivity), $p_{ij} := \overline{a_ia_j} \cap \overline{b_ib_j}$ (which exists because $a_i,b_i,a_j,b_j,p$ are coplanar and any noncollinear three of them, excluding $D$ if one of them is not ordinary, determine the plane), and $\lambda_{ij} := \overline{p_{ik}p_{jk}}$ where $\{i,j,k\} = \{1,2,3\}$.  The lines $\lambda_{ij}$ exist because if $p_{ij}=p_{ik}$ ($i,j,k$ all different), then this point is the intersection of $\overline{a_ia_j}$ and $\overline{a_ia_k}$ but that intersection is $a_i$, and it is also the intersection of $\overline{b_ib_j}$ and $\overline{b_ib_k}$ but that intersection is $b_i$, from which it follows that $a_i=b_i$, contrary to our standing assumption that all six vertices are distinct.

Now we observe that all points $p_{ij} \in \pi_A \cap \pi_B$, so all lines $\lambda_{ij} \subseteq \pi_A \cap \pi_B$.  But as we assumed $\pi_A \neq \pi_B$, their intersection cannot consist of more than one line.  It follows that $\lambda_{12} = \lambda_{13} = \lambda_{23}$ and this is the required axis of perspectivity.

NEW, ?\\
Let $l_i := \overline{a_ib_i}$, which exists and contains $p$ by central perspectivity.  Then $l_i \cup l_j$, which are ordinary lines that intersect in $p$, lie in a plane $\pi_{ij}$ and $p_{ij} := \overline{a_ia_j} \cap \overline{b_ib_j}$, which exists because $a_i,b_i,a_j,b_j,p$ are coplanar in $\pi_{ij}$ and any noncollinear three of them, excluding $D$ if one of them is not ordinary, determine the plane.  
All points $p_{ij} \in \pi_A \cap \pi_B$.  As we assumed $\pi_A \neq \pi_B$, their intersection cannot consist of more than one line; it is either an ordinary line or a subset of a special line.  This line is the required axis of perspectivity.
\end{proof}

Theorem \ref{thm:highercentraldesargues} reinforces our belief that a nontrivial projective rectangle should be regarded as, in a strange way, nonplanar.  Unfortunately, we were not yet able to make this intuition precise.

\sectionpage\section{The subplane construction}\label{sec:subplane}

We present a general construction which, at least sometimes, generates a projective rectangle in a projective plane, and put it to use to prove that many projective planes contain nontrivial projective rectangles.

\begin{construction}[Subplane Construction]\label{ex:subplaneconstruction}
Given a projective plane $\Pi$ and a subplane $\pi$.
Pick a point $D \in \pi$ and let $\cS$ be the set of all lines of $\Pi$ of the form $\overline{Dp}$ for $p \in \pi$ and $\cP_R = \bigcup \cS$, the set of all points of all lines in $\cS$.  Finally, let $\cO$ be the set of all restrictions to $\cP_R$ of lines of $\Pi$ that do not contain $D$.  We call the lines in $\cS$ \emph{long lines} and those in $\cO$ \emph{short lines}.  We shall call the incidence structure $R=(\cP_R,\cS\cup\cO)$ (incidence being containment) a \emph{pseudo-projective rectangle}.  It satisfies all the axioms of a projective rectangle except the essential Axiom (A\ref{Axiom:A6}).
\end{construction}

\begin{lem}\label{L:Dinsubplane}
For the pseudo-projective rectangle to be a projective rectangle, $D$ must be taken in the subplane $\pi$.
\end{lem}
\begin{proof}
Suppose $(\cP,\cS\cup\cO)$ is a projective rectangle.  Since $\pi \subseteq \cP$, $\pi$ is a plane of the projective rectangle; therefore it contains $D$ by Proposition \ref{prop:Dinpp}.
\end{proof}

We could simplify the construction:  Take a subplane $\pi$ and one line $l_0$ in it, and any point $D$ in $\pi \setminus l_0$.  For the projective rectangle, take all lines of $\Pi$ that join $D$ to $l_0$ and for $\cP_R$ take all points of $\Pi$ on those lines.  This gives precisely the subplane construction, because already it gives all the points of $\pi$ and then only the points generated from $D$ and $\pi$ in that construction.

We want the subplane construction to give rise to a projective rectangle.  When it does, the special lines of the $\PR$ are the long lines and the ordinary lines are the short lines.  Unfortunately, we have not found a proof (or disproof), so we present a restricted proof.  The proof we present here calls upon a theorem from the sequel \cite{pr3} about harmonic conjugation.  This proof is only valid in a plane that has harmonic conjugation and when the subplane is prime, i.e., isomorphic to $\PP(\bbZ_p)$ for a prime number $p$.  The planes with harmonic conjugation are the Moufang planes \cite[p.\ 202]{st}.  Since all finite Moufang planes are Pappian, that is, coordinatized by a (commutative) field \cite{st}, any finite subplane of a Moufang plane is Pappian and a minimal finite subplane of a Moufang plane is prime.  (An entirely different proof for all finite projective planes, based on graph theory, appears in the sequel \cite{pr2a}.)

\begin{thm}\label{T:subplaneconstruction}
The subplane construction from a prime subplane in a Moufang projective plane produces a projective rectangle.
\end{thm}

\begin{exam}
For use in the proof we define a matroid generalization of $L_2^k$ that we call $L_p^\bV$, where $p$ is a prime number and $\bV$ is a vector space of any dimension over $\bbZ_p$.  Let
$$
A:=\{ a_g \mid g \in \bV \} \cup \{D \},\ B:=\{ b_g \mid g \in \bV \}
\cup \{D \}, \text{ and } C:= \{ c_g \mid g \in \bV \} \cup\{D \}.
$$ 
All points $a_g, b_g, c_g, D$ are distinct.  
$L_p^\bV$ is the simple matroid of rank 3 on the ground
set $A\cup B\cup C$ whose rank-2 flats are the $3$ lines $A$, $B$, $C$ and the $p^{2\dim\bV}$ lines $\{a_g, b_{g+ h}, c_h \}$ with $g$ and $h$ in $\bV$.  

(This matroid is the complete lift matroid $L_0(\bV{\cdot}K_3)$ in \cite[Section 3]{b2}.)
\end{exam}

\begin{proof}
The subplane construction produces a family of lines through $D$ whose point set $\cP_R$ constitutes the putative projective rectangle.  The Moufang plane $\pi$ is coordinatized by an alternative ring that is an algebra over the coordinate field $\bbZ_p$ of $\pi$; we let $\bV$ denote the $\bbZ_p$-vector space of that ring.  Any three lines through $D$ in the subplane $\pi$ form a matroid isomorphic to $L_p^\bV$ (\cite[Section 4]{b4}).  According to \cite[Theorem 5.1]{pr3}, by harmonic conjugation in $\pi$ we obtain a projective rectangle $\PR$ whose point set is the harmonic closure of $A \cup B \cup C$, whose short lines have order $p$ (that is, $p+1$ points) and whose long lines are the lines $\overline{pD}$ of $\pi$ for $p \in \pi\setminus D$.
\end{proof}

\begin{exam}[Some Infinite Projective Rectangles]\label{ex:infinitesubplane}
Consider an infinite projective plane that contains a countably infinite subplane, such as $\PP(\bbR)$ with $\PP(\bF)$ for any countably infinite field $\bF$.
Theorem 6.6 of \cite{pr3} tells us that the harmonic closure of $\PP(\bbQ)$ is a projective rectangle $\PR$ of order $(|\bF|,|\bbR|)$.  The order of $\PR$ is independent of the choice of $\bF$ but the projective planes in $\PR$ are Pappian with coordinate field $\bF$, which is certainly not independent of $\bF$.  These examples demonstrate that the order of an infinite projective rectangle does not determine the rectangle.

A similar conclusion holds more generally for any uncountably infinite Moufang plane that contains $\PP(\bF)$ as a subplane.  
\end{exam}

We found another way to state the subplane construction when it does give a projective rectangle.  It can be regarded as an alternative axiom system for projective rectangles inside projective planes.

\begin{construction}[Fan Construction]\label{ex:fan-construction}
Let $\PP = (\cP, \cL, \cI)$ be a projective plane, $D$ a point in $\PP$, and $\cS$ any subset of the lines on $D$.  Let $R = (\cP_R, \cL_R, \cI_R)$ be the incidence structure with point set $\cP_R = \bigcup_{s \in \cS} s$, line set $\cL_R = \cS \cup \cO_R$ where $\cO_R = \{l \cap \cP_R:  l \in \cL \setminus \cS\}$, and incidence relation as in $\PP$.
\end{construction}

\begin{thm}\label{T:sub-construction}
The system $R$ is a projective rectangle if and only if it satisfies the following properties:
\begin{enumerate}[{\rm (R1)}]
\item $\cS$ contains at least $3$ lines.
\item If two lines $l_1, l_2 \in \cO_R$ have a common point, then $R$ contains a projective plane of which $l_1$ and $l_2$ are lines.
\end{enumerate}
\end{thm}

First, we show that (R1) can be weakened.

\begin{lem}\label{lem:R1-3}
Property (R1) can be replaced by 
\begin{enumerate}[{\rm (R1)}]
\item  [{\rm (R1$'$)}]  $\cS$ contains at least $2$ lines.
\end{enumerate}
\end{lem}

\begin{proof}
Suppose $\cS$ contains (at least) two lines on $D$.  Let $p$ be in one line of $\cS$ and $q_1,q_2$ in another, neither one the special point $D$; then the lines $l_1=\overline{pq_1} \cap \cP_R$ and $l_2=\overline{pq_2} \cap \cP_R$ are in $\cO_R$ with a common point.  By assumption, $R$ contains a projective plane $\pi$ in which $l_1$ is a line.  A line in a projective plane has at least three points.  $\overline{pq_1}$ intersects each line $s \in \cS$ in a point of $\PP$, each of which is in $\cP_R$ hence in $l_1$, so $\cS$ must contain at least three lines.
\end{proof}

\begin{proof}[Proof of Theorem \ref{T:sub-construction}.]
We review the axioms.
\begin{enumerate}[({A}1)]
\item Two points in $\cP_R$ generate a line $l$ of $\PP$.  If $l \in \cS$, the points are collinear in $R$.  If $l \notin \cS$, the points are collinear in $l \cap \cP \in \cO_R$.
\item Two lines $s_1,s_2 \in \cS$ each contain two points other than $D$, which make four points of which no three are collinear.
\item is implied by Lemma \ref{lem:R1-3}.
\item By its definition, $\cP_R$ contains $D$.
\item is implicit in the proof of Lemma \ref{lem:R1-3}.
\item Let $\pi$ be the projective plane contained in $\cP_R$ of which $l_1$ and $l_2$ are lines.  If $l_3$ or $l_4$ is in $\cS$, the conclusion of (A6) follows from (A5).  Otherwise, $l_3$ and $l_4$ are lines of $\pi$ and therefore intersect in a point of $\pi \subseteq \cP_R$, which is the conclusion of (A6). 
\qedhere
\end{enumerate}
\end{proof}

The difficulty of applying Theorem \ref{T:sub-construction} is that (R2) is hard to verify in examples.

The fan construction is a special case of the subplane construction.  Let $\pi$ be any plane of $R$.  Since $R$ is a projective rectangle, $D \in \pi$ (Proposition \ref{prop:Dinpp}).  The subplane construction applied to $D$ and $\pi$ gives $R$.

\begin{problem}
We believe the subplane construction gives a projective rectangle in every Desarguesian projective plane.  We have proofs for certain cases (Moufang planes, finite Dearguesian planes), but we really want a proof based on incidence geometry.  This is an open problem.
\end{problem}

\sectionpage\section{Narrow rectangles}\label{sec:narrow}

The smallest allowed value of $m+1$ is 3.  We call a projective rectangle \emph{narrow} if it has $m=2$.  We classify the narrow projective rectangles using some matroid theory.

A matroid like $L_2^k$ of Example \ref{ex:L2k} is defined for any nontrivial quasigroup $\mathfrak{G}$, simply replacing $\bbZ_2^k$ by $\mathfrak{G}$; it is the complete lift matroid $L_0(\mathfrak{G}K_3)$ from \cite{b2} or \cite{bgpp}).  We define it in a way compatible with Example \ref{ex:L2k}.  The ground set is $E:= A\cup B\cup C$ where $A:= \left\{ a_g \mid g \in \mathfrak{G} \right\}
\cup \{D \}$, $B:= \left\{ b_g \mid g \in \mathfrak{G} \right\} \cup \{D \}$ and
$C:= \left\{ c_g \mid g \in \mathfrak{G} \right\} \cup \{D \}$.  The lines (rank-2 flats of the matroid) are $A$, $B$, and $C$ and the sets $\{a_g, b_{g h}, c_h \}$ with $g, h \in \mathfrak{G}$.  If this is a projective rectangle, $A$, $B$, and $C$ are the special lines and the other lines are the ordinary lines.  If $L_0(\mathfrak{G}K_3)$ is a projective rectangle, it is narrow, and $\mathfrak{G}$ can only be certain groups.

\begin{prop}\label{prop:narrow}
Every narrow projective rectangle $\PR$ has the form $L_0(\mathfrak{G}K_3)$ where $\mathfrak{G}$ is a nontrivial group with exponent $2$, and conversely.  If\/ $\PR$ is finite the group is $\mathbb{Z}_2^k$ with $k\geq1$ and its parameters are $(m,n)=(3,2^k+1)$ with $k\geq1$.
\end{prop}

This proposition includes infinite groups.  

\begin{proof}
First we note that every narrow projective rectangle $\PR$ is an $L_0(\mathfrak{G}K_3)$ where $\mathfrak{G}$ is a quasigroup of order greater than 1.  There are three special lines, which we call $A$, $B$, and $C$.  We label the elements of each line, except $D$, by a set $\mathfrak{G}$ of labels and we define an operation on $\mathfrak{G}$ by $gh=k$ such that $a_gc_hb_k$ is an ordinary line of\/ $\PR$.  It is clear that this is well defined and that any two of $g,h,k$ determine the third, so $\mathfrak{G}$ is a quasigroup.  Then $\PR$ is the same as $L_0(\mathfrak{G}K_3)$.

Now let $\mathfrak{G}$ be a quasigroup, and assume that $L_0(\mathfrak{G}K_3)$ is a projective rectangle.  We prove that $\mathfrak{G}$ satisfies the following fundamental property:
\begin{equation}
\label{eq:quasirectangle}
gh=ef \implies gf=eh.
\end{equation}
Consider the lines $l_1=\{a_g,b_{gh},c_h\}$ and $l_2=\{a_e,b_{ef},c_f$ in Axiom (A\ref{Axiom:A6}), and two other lines, $l=\{a_g,b_{gf},c_f\}$ and $l'=\{a_e,b_{eh},c_h\}$.  According to Axiom (A\ref{Axiom:A6}) the lines $l$ and $l'$ should have a common point, so $b_{gf}=b_{eh}$, which means $gf=eh$.

The matroid structure of $L_0(\mathfrak{G}K_3)$, hence the structure of the corresponding projective rectangle, is not affected by isotopy of $\mathfrak{G}$ because it is entirely determined by the 3-point lines $\{a_g, b_{g h}, c_h \}$, in which isotopy only permutes subscript names without changing their algebraic relation.  Thus, we may replace $\mathfrak{G}$ by any convenient isotope of itself.  In particular, any quasigroup is isotopic to a loop (a quasigroup with identity element, 1), so we may assume $\mathfrak{G}$ is a loop.  Suppose $h=e=1$ in Equation \eqref{eq:quasirectangle}.  Then $g=f \implies gf=1$; in other words, $gg=1$ for every element of $\mathfrak{G}$.  Suppose $g=h$ and $e=f$.  Then $1=1 \implies ge=eg$; that is, $\mathfrak{G}$ is commutative.  A property that characterizes a quasigroup that is isotopic to a group is the Quadrangle Criterion \cite{dk}, which is
$$
\left.\begin{array}{l}
a_1c_1=a_2c_2\\
a_1d_1=a_2d_2\\
b_1c_1=b_2c_2
\end{array} \right\}
\implies b_1d_1=b_2d_2.
$$
We prove the Quadrangle Criterion for $\mathfrak{G}$ by means of Equation \eqref{eq:quasirectangle}.  
\begin{align*}
a_1c_1=a_2c_2 &\implies a_1a_2=c_1c_2, \\
a_1d_1=a_2d_2 &\implies a_1a_2=d_1d_2, \\
b_1c_1=b_2c_2 &\implies b_1b_2=c_1c_2.
\end{align*}
The first two lines imply that $c_1c_2=d_1d_2$ and combined with the third line we deduce that $b_1b_2=d_1d_2$, proving the Quadrangle Criterion.  Hence, $\mathfrak{G}$ is isotopic to a group.  By isotopy we may assume $\mathfrak{G}$ is a group, and since a group is a loop, the group is abelian and has exponent 2 (hence may be written additively as in Example \ref{ex:L2k}).  If $\mathfrak{G}$ is finite, it is $\bbZ_2^k$ for some positive integer $k$ as in Example \ref{ex:L2k}.  These necessary properties of $\mathfrak{G}$ are sufficient for $L_0(\mathfrak{G}K_3)$ to be a projective rectangle, because exponent 2 implies Axiom (A\ref{Axiom:A6}), as is easy to verify.
\end{proof}

The geometry of a narrow projective rectangle is determined by the isotopy type of its quasigroup.  Thus, the finite such rectangles are obtained from a finite Pappian projective plane of 2-power order by the subplane construction of Section \ref{sec:subplane} using a Fano subplane.

\sectionpage\section{Orthogonal arrays from projective rectangles}\label{S:OA}

A \emph{transversal design} is a partition of a set $\cP_T$ of $(m+1)n$ points into $m+1$ special sets of size $n$ together with a family of $m+1$-subsets of $\cP_T$ such that each such $m+1$-set intersects each special set exactly once and each pair of points not contained in a special set lies in exactly one $m+1$-set.  A projective rectangle with $D$ deleted is exactly a transversal design with the extra partial Pasch property Axiom (A\ref{Axiom:A6}).  A dual concept to transversal designs is that of orthogonal arrays; the corresponding dual to projective rectangles is orthogonal arrays with a dual property to (A\ref{Axiom:A6}).  We explore that dual concept in this section.\footnote{We thank Douglas Stinson for drawing our attention to transversal designs.}

An orthogonal array (OA) is a generalization of orthogonal latin squares. We adopt the notation for orthogonal arrays used in \cite{Hedayat}.
An $N\times k$ array $A$ with entries from $S$ (a set of size $s$) is said to be an \emph{orthogonal array}, $OA_{\lambda}(N,k,s,t)$, with $s$ symbols, strength $0\le t \le k$,
and index $\lambda$ if every $N\times t$ subarray of $A$ contains each $t$-tuple based on $S$ exactly $\lambda$ times as a row.  We write $a(r,c)$ for the label that appears in row $r$ and column $c$.

\subsection{An orthogonal array from points and lines}\label{S:OAlines}\

In order to represent a projective rectangle $\PR$ as an orthogonal array of points and lines, we formulate a special property for an orthogonal array of type $OA_{1}(n^2, m+1, n,2)$.
\begin{enumerate}
\item[(OA6)]  If four rows in the orthogonal array appear like the first five columns $c_{ij}$ in this table, 
\begin{center}
\begin{tabular}{c|ccccccc}
	&$c_{12}$	&$c_{13}$	&$c_{24}$	&$c_{14}$	&$c_{23}$	&$c_{34}$	\\
\hline
$r_1$	&$a_{12}$&$a_{13}$&	&$a_{14}$&	&	\\
$r_2$	&$a_{12}$&	&$a_{24}$&	&$a_{23}$&	\\
$r_3$	&	&$a_{13}$&	&	&$a_{23}$&$a_{34}$\\
$r_4$	&	&	&$a_{24}$&$a_{14}$&	&$a_{34}$	
\end{tabular}
\end{center}
where it is possible that $c_{13}=c_{24}$ or $c_{14}=c_{23}$, then there is a sixth column that appears like $c_{34}$.
(The empty cells are arbitrary.)
\end{enumerate}
The property (OA6) does not follow from the definition of an orthogonal array.  We are not aware that it has been considered in the theory of orthogonal arrays or dually in transversal designs.  Its contrary, that the sixth column of (OA6) never appears, arises (in the language of transversal designs) as the ``anti-Pasch configuration'' in \cite{dls} (whose ``Pasch configuration'' is slightly stricter than ours).\footnote{We are very grateful to Charles Colbourn for hunting in the literature and communicating these facts.}

\begin{thm}\label{OA:lines}
Let $n\geq m \geq 2$.  
\begin{enumerate}[{\rm(i)}]
\item A projective rectangle $\PR$ of order $(m,n)$ gives rise to an orthogonal array $OA_{1}(n^2, m+1, n,2)$ with property {\rm(OA6)}.
\item An orthogonal array $OA_{1}(n^2, m+1, n,2)$ gives rise to a projective rectangle $\PR$ of order $(m,n)$ if, and only if, it satisfies the additional property {\rm(OA6)}.
\end{enumerate}
\end{thm}

We note that Part (ii) is a strengthening of the converse of Part (i).

\begin{proof}
We begin by proving Part (i).  In $\PR \setminus D$ we have $m+1$ special lines partitioning all the points, and $n^2$ ordinary lines.  
By Theorem \ref{numberofpoinsinlines}, every ordinary line intersects every special line exactly once and every pair of points in different special lines lie in exactly one ordinary line. 
Each ordinary line will give a row of the orthogonal array and each special line will give a column.  We label the points in each special line by the numbers $1,\dots,n$ and we write $a(p)$ for the label of the point $p$.  The entries in a row are the labels of the points that appear in that ordinary line, arranged in the column of the special line that contains the point.  Thus, each pair of labels appears once in each pair of columns.  
That is a 2-$(n,m+1,1)$ orthogonal array in standard notation.  
In  the notation used in \cite{Hedayat}, it is an $OA_{1}(n^2, m+1, n,2)$.

Property (OA6) is the interpretation of Axiom (A\ref{Axiom:A6}) for an $OA_{1}(n^2, m+1, n,2)$.  In Axiom (A\ref{Axiom:A6}) let $l_3$ and $l_4$ be the two lines besides $l_1$ and $l_2$.  The assumption in the axiom is that points $p_{ij} = l_i \cup l_j$ exist for $(i,j) = (1,2),(1,3),(2,4),(1,4),(2,3)$.  Let $s_{ij}$ be the special line that contains $p_{ij}$; we note that the special lines are distinct except that $s_{13}$ may be the same as $s_{24}$ and $s_{14}$ may be the same as $s_{23}$.  In the orthogonal array derived from $\PR$, the row of line $l_i$ is $r_i$, the column of line $s_{ij}$ is $c_{ij}$, and the label of $p_{ij}$ is $a(r_i,c_{ij})=a(r_j,c_{ij})$.    Therefore, the array looks as in Property (OA6), except for the last column.

The conclusion of Axiom (A\ref{Axiom:A6}) is that there is a point $p_{34}$ that is incident with both lines $l_3$ and $l_4$.  That translates to the existence of a final column as in (OA6) with $a_{34} = a(p_{34})$.  Hence, Property (OA6) is satisfied by the array derived from the projective rectangle $\PR$.

Proof of Part (ii).  Suppose we have an $OA_{1}(n^2, m+1, n,2)$.  Let $C$ be the set of $m+1$ columns, let $R$ be the set of rows, let $L$ be the set of $n$ labels in the array, and write $a(r,c)$ for the entry in row $r$, column $c$.  We form an incidence structure whose point set is $(C\times L) \cup  D$.  The lines of this structure are special lines, of the form $s_c = \{(c,a) : a \in L \} \cup D$, for each $c\in C$, and ordinary lines, of the form $l_r = \{(c,a) : c \in C \text{ and } a= a(r,c) \}$, for each $r\in R$.

We prove this incidence structure satisfies Axioms (A\ref{Axiom:A1})--(A\ref{Axiom:A5}) of a projective rectangle.  We assumed $n\geq m\geq2$ so in the orthogonal array there are at least two distinct labels, which we call $a_1$ and $a_2$, and at least $3$ columns, of which three are $c_1,c_2,c_3$.  There are also at least $2^3$ rows.

Proof of Axiom (A\ref{Axiom:A1}).  We consider two points $p_1=(r_1,a_1)$ and $p_2=(r_2,a_2)$ where $a_1=a(r_1,c_1)$ and $a_2=a(r_2,c_2)$.  
The points belong to the same special line if and only if $c_1=c_2$.  The special line is $s_{c_1}$.  Otherwise, there is exactly one row $r$ where the entry in column $c_1$ is $a_1$ and the entry in column $c_2$ is $a_2$.  Then $p_1$ and $p_2$ belong to the ordinary line $l_r$.

Proof of Axiom (A\ref{Axiom:A2}).  Among the three pairs $a(r_1,c_j), a(r_2,c_j)$ for $j=1,2,3$, only one can be the same label, $a(r_1,c_j) = a(r_2,c_j)$, because each ordered pair of labels appears only once in the same two columns.  Say $a(r_1,c_1) \neq a(r_2,c_1)$ and $a(r_1,c_2) \neq a(r_2,c_2)$.  Then $(c_1,a(r_1,c_1)), (c_1,a(r_2,c_1)), (c_2,a(r_1,c_2)), (c_2,a(r_2,c_2))$ are four points, no three collinear.

Proof of Axiom (A\ref{Axiom:A3}).  The special line $s_c$ contains at least the three points $D, (c,a_1), (c,a_2)$.  The ordinary line $l_r$ contains the points $(c_1,a(r,c_1)), (c_2,a(r,c_2)), (c_3,a(r,c_3))$.

Proof of Axiom (A\ref{Axiom:A4}).  This follows by the definition of the incidence structure.

Proof of Axiom (A\ref{Axiom:A5}).  Two special lines intersect only in $D$.  A special line $s_c$ and an ordinary line $l_r$ intersect only in the point $(c,a(r,c))$.

Finally, we prove Axiom (A\ref{Axiom:A6}) from Property (OA6).  Let $r_1, r_2$ be the rows of the array that correspond to the lines $l_1, l_2$ in this axiom and let $l_3,l_4$ be the two other lines with corresponding rows $r_3,r_4$.  The hypotheses of intersection imply that the diagram in Property (OA6) is satisfied, possibly except for the last column.  By the assumption of Property (OA6), the final column does exist.  This implies that $l_3\cap l_4$ is the point $p_{34}$ in the special line $s_{34}$ that corresponds to column $c_{34}$ and has the label $a(p_{34} = a_{34}$.  Therefore, the conclusion of Axiom (A\ref{Axiom:A6}) is satisfied.

If on the contrary there is a failure of Property (OA6), then the final column fails to exist in at least one instance and the corresponding configuration in the rectangle fails to have the intersection point promised by Axiom (A\ref{Axiom:A6}), so it is not a projective rectangle.
\end{proof}

\subsection{An orthogonal array from points and planes}\label{S:OAplanes}\

Ryser gives a nice construction of an orthogonal array from a projective plane \cite[p.~92]{Ryser}.
We extend Ryser's ideas to construct an orthogonal array from points and planes of a projective rectangle by partitioning the ordinary points outside a given ordinary line by means of the separate planes that contain that line.
The proof is based on the proof that Ryser gives for projective planes, adapted to the existence of multiple planes.

\begin{lem}\label{L:partition}
Let $l$ be an ordinary line in a finite $\PR$.  The family of sets $\pi\setminus (l\cup D)$ for all planes $\pi$ that contain $l$ is a partition of the points in $\PR\setminus (l \cup D)$ into $(n-1)/(m-1)$ parts of\/ $(m+1)(m-1)$ points each.
\end{lem}

\begin{proof}
We observe that every plane in $\PR$ containing $l$ also contains the special point $D$.  
If $p\not\in l \cup D$, then by Corollary \ref{cor:3pointssubplane} there is a unique plane on $l$ that contains $p$; thus, the planes on $l$ partition  the points in $\PR\setminus (l \cup D)$.  The number of such planes is given by Theorem \ref{thm:counting} Part \eqref{prop:counting:pionl}.
The number of parts of the resulting partition equals the number of planes that contain the line $l$.
\end{proof}

\begin{thm}\label{OA:Ryser} Suppose that $(m,n)$ is the order of the projective rectangle $\PR$.  
Let $l \in \PR$ be an ordinary line and let $\pi_1, \pi_2, \dots, \pi_w$ be all the planes in $\PR$ that contain $l$, where $w=(n-1)/(m-1)$. Then
 $\PR$ gives rise to an orthogonal array of the form $OA_{w}(m^2w, m+1, m,2)$.
\end{thm}

\begin{proof}  Let $p_1, p_2, \dots, p_m$ be the points of $l$.  We label the points in $\pi_{i}\setminus l$ by $q_{1}^{i}, q_{2}^{i}, \dots, q_{k}^{i}$ where $k=m^2$
($D$ is one of these points) and label the lines on $p_r$ in $\pi_{i}\setminus l$ with $1, 2, \dots, m-1$ for each $r=1,2, \dots, m$.
We write $a_{st}^{i}$ to record the label of the line $q_{s}^{i}p_{t} \in \pi_{i}$.

We claim that the matrix $A_{i}=[a_{st}^{i}]_{s,t}$ is an orthogonal array of the form  $OA_1(m^2,m,m-1,2)$.
We prove this by contradiction. Suppose that there two ordered pairs in the rows of $A_{i}$ that are equal; that is, $(a_{s_1t_1}^{i},a_{s_1t_2}^{i}) =(a_{s_2t_1}^{i},a_{s_2t_2}^{i})$ with $s_1\ne s_2$.
Therefore,  $a_{s_1t_1}^{i}=a_{s_2t_1}^{i}$ and  $a_{s_1t_2}^{i} =a_{s_2t_2}^{i}$. The equality of these labels implies that the points $q_{s_1}^{i}$, $q_{s_2}^{i}$, and $p_{t_1}$
are collinear and that $q_{s_1}^{i}$, $q_{s_2}^{i}$, and $p_{t_2}$ are also collinear.  Thus, each $p_{t_j}$ is the unique point of $l$ on the same line $\overline{q_{s_1}^{i} q_{s_2}^{i}}$.  Therefore, $p_{t_1} = p_{t_2}$, but that is impossible because $t_1 \neq t_2$.

Now let $B=\begin{bmatrix} A_{1} \\ A_{2} \\ \vdots \\ A_{w} \end{bmatrix}$. This matrix is an orthogonal array of the form $OA_{\lambda}(m^2w,m+1,m,2)$ where $\lambda = \sum_{i=1}^w 1 = w$. That completes the proof.
\end{proof}

\begin{exam}\label{ex:ryser}
We give an example for Theorem \ref{OA:Ryser} using the projective rectangle $L_2^2$ depicted in Figure \ref{figure1}. For the sake of simplicity we pick the  
line $l=\{a_1, b_1,c_1\}$. We recall that for an ordinary line in $L_2^2$, there are exactly $\lambda=3$ planes having that line in common.   
Figure \ref{OrthogonalArray} shows the three planes embedded in $L_2^2$ with $l$ as common line. 

For the first plane, let's say $\pi_{1}$, we distinguish the points  
$a_1$, $a_g$, $b_1$, $b_g$, $c_1$, $c_g$ and $D_1:=D$. For a fixed point in $l$ theres two lines in $\pi_{1}\setminus l$ passing by the fixed point; from the set $\{1,2\}$ we assign labels to these lines.  
For the lines $\{a_1,a_g,D_1\}$ and $\{a_1,b_g,c_g\}$, which intersect $l$ at $a_1$, we assign $1$ and $2$ to them, respectively.  We arbitrarily assign $1$ and $2$ to 
$\{b_1,b_g,D_1\}$ and $\{a_g,b_g,c_g\}$, respectively, and also to $\{a_g,b_g,c_1\}$ and  $\{c_g,c_1,D_1\}$.  
With these labels we construct the first four rows of the rectangular array in Table  \ref{ExRectanArray}. 
The columns of the array are labeled on top with the points in the line 
$l$ and the rows are labeled on the left with the points in each plane that are not in $l$. 
In this case the first four rows are 
labeled with the points in  $\pi_{1}\setminus l$. The entries of the  rectangular array are the labels of the lines passing through the point in the column label and the point  
in the row label. For instance, the first entry of the first row in Table \ref{ExRectanArray} is $1$, because the line passing through $a_1$ and $a_g$ has label $1$.  
The first entry of the fourth row is $1$, because the line passing through $a_1$ and $D$ has label $1$. 

The second plane in Figure \ref{OrthogonalArray}, $\pi_2$, has the points $a_1$, $a_h$, $b_1$, $b_h$, $c_1$, $c_h$ and $D_2:=D$.  
As in $\pi_{1}$, we assign arbitrary labels from $\{1,2\}$.  We choose $1$ to be 
the label of $\{a_1,b_h,c_h\}$,  $\{a_h,b_1,c_h\}$, and $\{c_1,c_h,D_2\}$ and $2$ as the label of $\{a_1,a_h,D_2\}$,  $\{b_1,b_h,D_2\}$, and $\{a_h,b_h,c_1\}$.  

For the third plane in Figure \ref{OrthogonalArray}, $\pi_3$ with points $a_1$, $a_{g+h}$, $b_1$, $b_{g+h}$, $c_1$, $c_{g+h}$ and $D_3:=D$, we also assign arbitrary labels from $\{1,2\}$. So, for example, $1$ will be 
the label of $\{a_1,a_{g+h},D_3\}$,  $\{a_{g+h},b_1,c_{g+h}\}$, and $\{a_{g+h},b_{b+h},c_1\}$ and $2$ will be the label of $\{a_1,b_{g+h},c_{g+h}\}$,  $\{b_1,b_{g+h},D_3\}$, and $\{c_1,c_{g+h},c_1\}$.  

These give the orthogonal array $OA_3(12,3,2,2)$. This is a $12 \times 3$ array filled with $2$ symbols, such that in any $2$ columns there are 4 different ordered pairs, each repeated $\lambda=3$ times.

\begin{table}[ht]
\begin{tabular}{c|ccc}
   	&$a_1$	&$b_1$	&$c_1$\\\hline
   $a_g$	& $1$	& $2$	&$1$\\
   $b_g$	& $2$	& $1$	&$1$\\
   $c_g$	& $2$	& $2$	&$2$\\
   $D_1$ 	& $1$	& $1$	&$2$\\
   $a_h$	& $2$	& $1$	&$2$\\
   $b_h$	& $1$	& $2$	&$2$\\
   $c_h$	&$1$	& $1$	&$1$\\
   $D_2$ 	& $2$	& $2$	&$1$\\
   $a_{g+h}$& $1$	& $1$	&$1$\\
   $b_{g+h}$& $2$	& $2$	&$1$\\
   $c_{g+h}$& $2$	& $1$	&$2$\\
   $D_3$ 	& $1$	& $2$	&$2$\\
\end{tabular}
\medskip
\caption{Orthogonal array $OA_3(12,3,2,2)$ for $L_2^2$.} \label{ExRectanArray} 
\end{table}
				
\begin{figure} [htbp]
	\includegraphics[width=60mm]{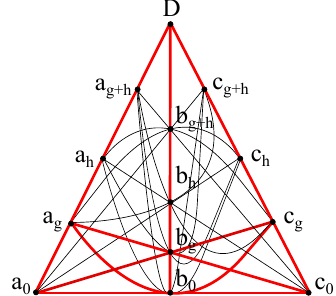} \quad 
	\raisebox{.6cm}{
	\includegraphics[width=30mm]{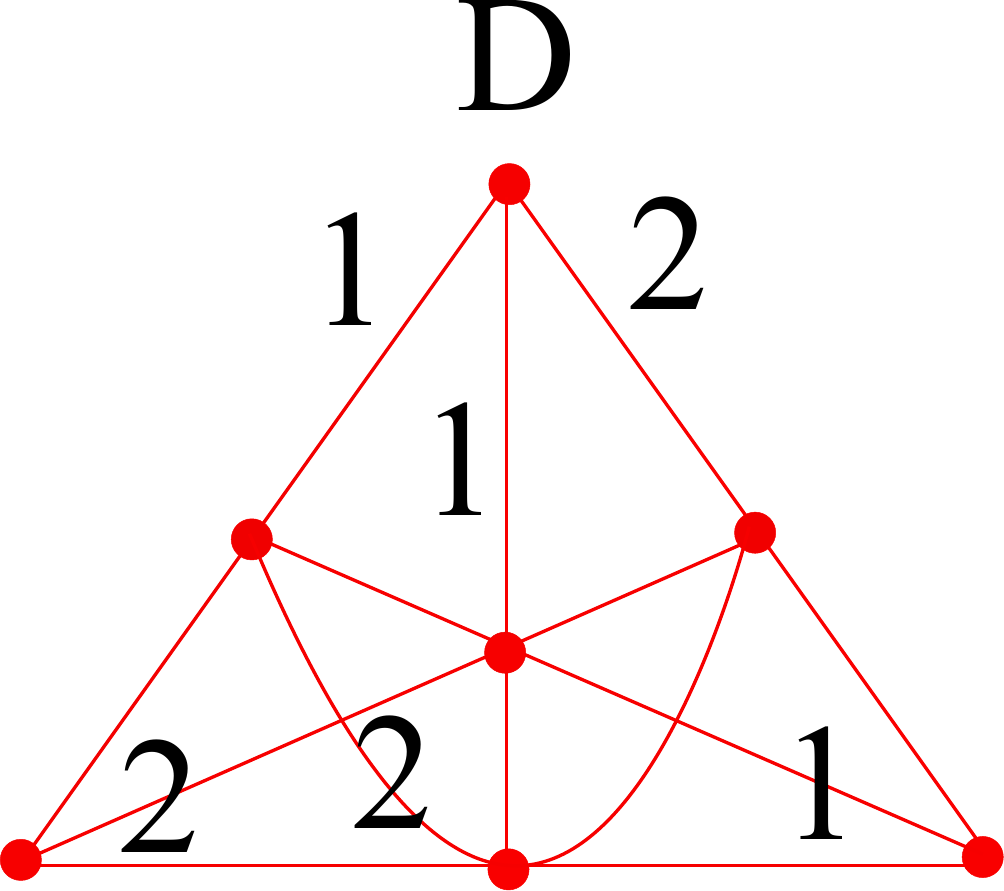}} \quad
	\raisebox{1.6cm}{
	\begin{tabular}{c|ccc}
   &$a_1$	&$b_1$	&$c_1$\\\hline
   $a_g$	& $1$	& $2$	&$1$\\
   $b_g$	& $2$	& $1$	&$1$\\
   $c_g$	& $2$	& $2$	&$2$\\
   $D_1$ 	& $1$	& $1$	&$2$\\
\end{tabular}}

	\bigskip
	
           \includegraphics[width=60mm]{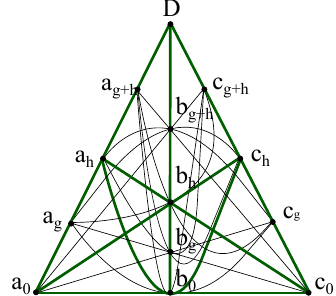} \quad 
	\raisebox{.6cm}{
	\includegraphics[width=30mm]{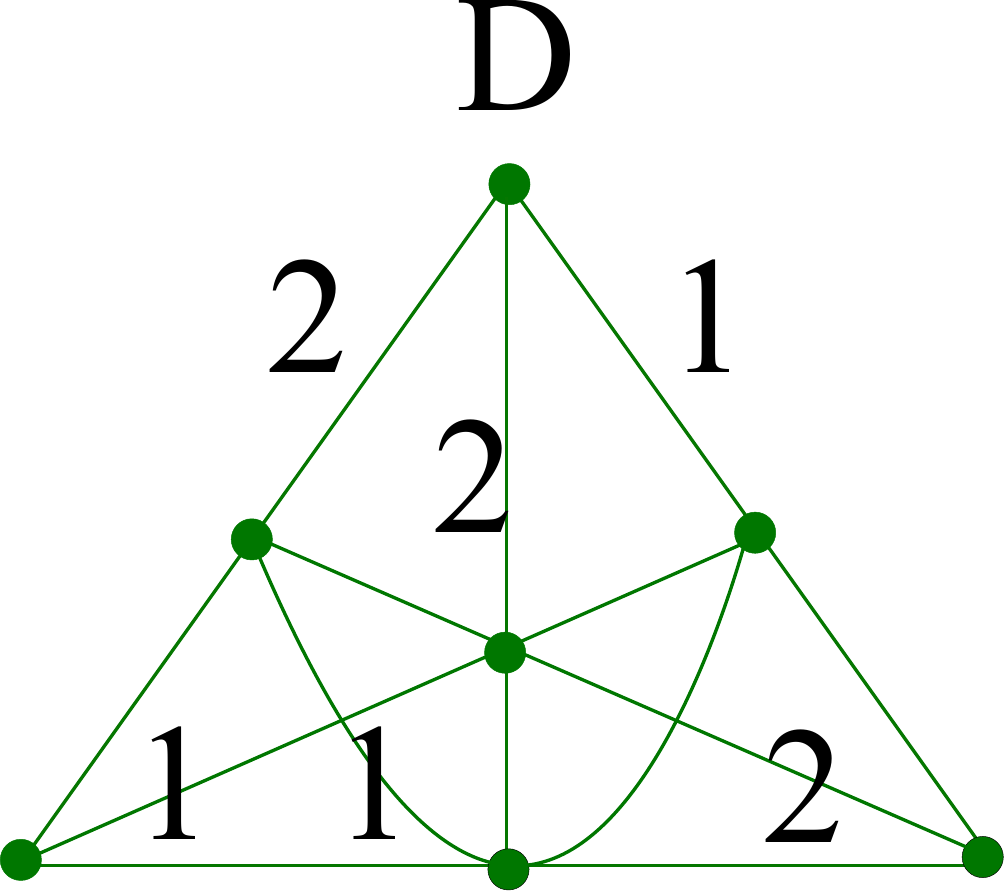}} \quad
	\raisebox{1.6cm}{
           \begin{tabular}{c|ccc}
   	&$a_1$	&$b_1$	&$c_1$\\\hline
   $a_h$	& $2$	& $1$	&$2$\\
   $b_h$	& $1$	& $2$	&$2$\\
   $c_h$	&$1$	& $1$	&$1$\\
   $D_2$ 	& $2$	& $2$	&$1$\\
\end{tabular}}

           	\bigskip

           \includegraphics[width=60mm]{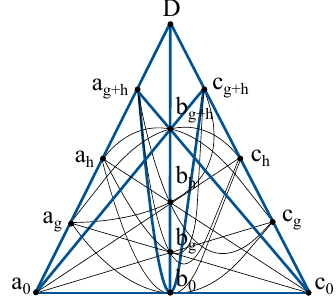} \quad 
	\raisebox{.6cm}{
	\includegraphics[width=30mm]{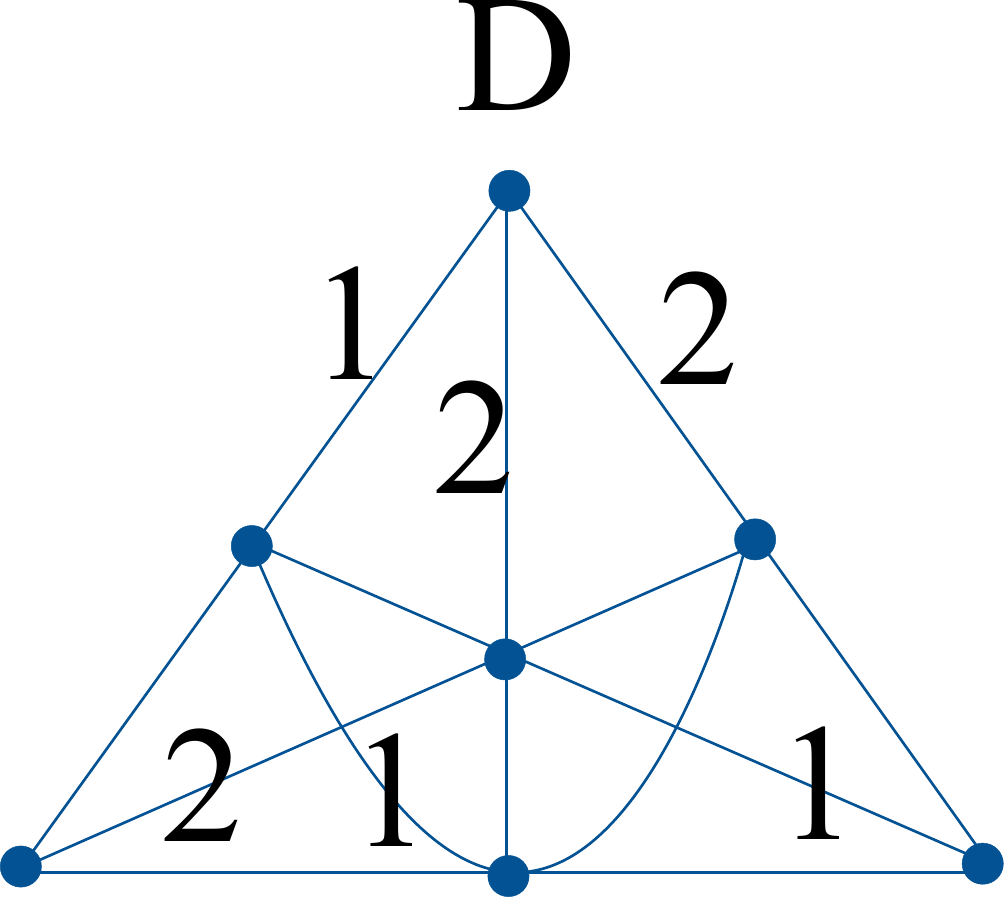}} \quad
	\raisebox{1.6cm}{
           \begin{tabular}{c|ccc}
   	&$a_1$	&$b_1$	&$c_1$\\\hline
   $a_{g+h}$& $1$	& $1$	&$1$\\
   $b_{g+h}$& $2$	& $2$	&$1$\\
   $c_{g+h}$& $2$	& $1$	&$2$\\
   $D_3$ 	& $1$	& $2$	&$2$\\
\end{tabular}}

\caption{Three views of the projective rectangle $L_2^2$.  In each view the lines of one plane $\pi_i$ are highlighted.  Next to that view is the OA $A_i$ corresponding to that plane with a small diagram showing the line labels.} \label{OrthogonalArray}
\end{figure}

\end{exam}

\newpage{\ }
\newpage{\ }

\sectionpage\section{The dual incidence structure}\label{sec:dual}

The dual structure is obtained by interchanging the roles of points and lines.  It is interesting in its own right, as it connects projective rectangles with incidence geometry in a different way.  The dual is essentially a net with a complete quadrangle property.  Being a dual projective rectangle, it contains all the dual projective planes of the planes of the original projective rectangle.

A \emph{net} $\sN$ is an incidence structure $(\cP,\cL,\mathcal{I})$ which consists of a set $\cP$ of points and a set $\cL$ of parallel classes $\cL_i$ ($i \in$ an index set) of lines, such that each line is a set of points, every point belongs to exactly one line of each parallel class, and any two lines of different parallel classes have exactly one point in common.  The theory of nets is extensive.  It is easy to prove that every parallel class has the same number of lines and that the number of points on every line is the same.

We call these points and lines \emph{ordinary}.  By adding a \emph{special point} for each parallel class, which is defined to belong to all lines of that class and no other ordinary lines, and adding one \emph{special line} that contains all the special points, we get a \emph{projectively extended net}.  (``Projectively'' refers to the existence of the special line.)

Two points might not be in any common line.  They are called \emph{collinear} if they are in a line.  They cannot be in more than one line.

We state the dualized rectangle axioms.

\begin{enumerate} [({A}1*)]
\item \label{Axiom:A1*}   Every two distinct lines contain exactly one point in common.

\medskip

\item \label{Axiom:A2*}  There exist four lines in the extended net with no three of them concurrent.

\medskip

\item \label{Axiom:A3*}  Every point is in at least three distinct lines.

\medskip

\item \label{Axiom:A4*}   There is a \emph{special line} $D^*$.
(A point in $D^*$ is called \emph{special}.  A point that is not in $D*$ and a line that is not $D^*$ are called \emph{ordinary}.)

\medskip

\item \label{Axiom:A5*}   Each special point belongs to exactly one line with each other point.

\medskip

\item \label{Axiom:A6*}  If two ordinary points $P_1$ and $P_2$ are collinear, then any two other points $P_1'$ and $P_2'$ that are collinear with $P_1$ and $P_2$ through four distinct lines (i.e., there are four distinct lines $P_iP_j'$ for $i,j=1,2$), are themselves collinear.

\end{enumerate}

A \emph{complete quadrangle} in a net consists of 4 points, no three collinear, and 6 lines determined by them.  A \emph{nearly complete quadrangle} consists of the same 4 points and 5 of the 6 lines, the 6th line possibly existing or not existing.
The dual Axiom (A\ref{Axiom:A6}*) can be stated (more elegantly) in terms of quadrangles:
\begin{enumerate}
\item[(CQP)]  (Complete Quadrangle Property)  Every nearly complete quadrangle is complete.
\end{enumerate}

\begin{lem}\label{L:cqp-net}
A net has the complete quadrangle property if and only if its projective extension satisfies (A\ref{Axiom:A6}*).
\end{lem}

\begin{proof}
Let $\sN$ be the net and $\overline\sN$ its projective extension.

Suppose $\sN$ satisfies (CQP).  In Axiom (A\ref{Axiom:A6}*), if $P_1'$ is special, it is in a line with every other point including $P_2'$, so we only need to consider ordinary points $P_1, P_2, P_1', P_2'$.  Then (CQP) implies the line $\overline{P_1'P_2'}$ so (A\ref{Axiom:A6}*) is satisfied by $\overline\sN$.

Conversely, assume $\overline\sN$ satisfies Axiom (A\ref{Axiom:A6}*).  Let $P_1, P_2, P_1', P_2'$ be any four points in $\sN$ of which five pairs, including $P_1$ and $P_2$, are collinear.  (A\ref{Axiom:A6}*) implies the sixth pair is collinear, thus verifying (CQP) in $\sN$.
\end{proof}

\begin{thm}\label{T:dual}
The dual of a projective rectangle is a projective extension of a net that has the complete quadrangle property, at least three parallel classes, and at least two lines in each parallel class; and conversely such a projective extension is the dual of a projective rectangle.
\end{thm}

\begin{proof}
First, we show that the dualized rectangle axioms imply a projective extension of a net $\sN$ with the three stated properties.  Each point of $\PR$ becomes a line in $\sN$ and each line becomes a point.  
A parallel class in $\sN$ is the set of lines dual to the ordinary points of a special line $s$ of $\PR$; thus it is the set of ordinary dual lines that contain a fixed special point $s^*$; the lines are parallel because of Axiom (A\ref{Axiom:A1*}*).  
There are at least three parallel classes because there is one for each special point and $\PR$ has at least three special points by Theorem \ref{numberofpoinsinlines} Part \eqref{numberofpoinsinlines:a}.  There are at least two lines in a parallel class because, by Theorem \ref{numberofpoinsinlines} Part \eqref{numberofpoinsinlines:g}, every special line in $\PR$ has at least three points, two of which are ordinary and correspond to parallel lines in the net.

Second, we consider how the dualized rectangle axioms apply to a projective extension of a net that satisfies the (CQP) and the other properties in the theorem.

(A\ref{Axiom:A1*}*) is true by definition if one of the lines is the special line.  It is valid in the net except when the lines are parallel.  Parallel lines have a common point in the extension.

(A\ref{Axiom:A2*}*)  To find four lines in the extended net with no three concurrent, take the special line, three special points, and one ordinary line on each of the special points.  If the three ordinary lines are concurrent, replace one of them by a parallel line.
Alternatively, take two lines from each of two parallel classes.

(A\ref{Axiom:A3*}*)  The existence of three distinct lines on each point is equivalent for an ordinary point to the existence of at least 3 parallel classes, and for a special point to the existence of a parallel to each ordinary line.

(A\ref{Axiom:A4*}*)   is part of the definition of a projectively extended net.

(A\ref{Axiom:A5*}*)   is part of the definition of a projectively extended net.

(A\ref{Axiom:A6*}*)  Lemma \ref{L:cqp-net} shows that the (CQP) of $\sN$ implies this axiom for the extended net.

\end{proof}

\sectionpage\section{Open problems}\label{sec:open}

Our work on nontrivial projective rectangles leaves many unanswered questions.  Here are some to add those in the body of the paper.

\begin{enumerate}[Q1.]

\item  All our examples of projective rectangles are substructures of Pappian projective planes that can be obtained by the subplane construction.  Are there other examples?

\item  We are ignorant of how a special line compares in its intersections with two planes $\pi$ and $\pi'$.  Two questions stand out.
\begin{enumerate}
\item If a plane $\pi$ has an ordinary line $l$, there are many other planes in which $l$ is a line.  However, if $l$ is special, i.e., $l = s \cap \pi$ for a special line $s$, we have no idea whether even one other plane has $l$ as a line.  
\item We do not know whether there may be another plane $\pi'$ such that $s \cap \pi \cap \pi'$ has a specific cardinality (not greater than $m+1$), what the possible values of $|s \cap \pi \cap \pi'|$ may be, whether 0 is a possible value in every nontrivial \PR\ (aside from $L_2^2$, where it is not), or in the infinite case whether it is even possible that $s \cap \pi'$ may properly contain $s \cap \pi$.
\end{enumerate}

\item  We proved the subplane construction of Section \ref{sec:subplane} only for Pappian planes, coordinatizable by a field.
\begin{enumerate}
\item Is there an analytic proof for skew fields?
\item Does an analytic proof using alternative algebras succeed in planes with weaker coordinate algebras such as near fields and alternative algebras?
\item Is there a synthetic proof for Pappian or Desarguesian or other projective planes?
\item Does the construction exist in non-Desarguesian, or non-Moufang, planes?
\end{enumerate}

\item  Are all planes in a projective rectangle isomorphic?  We were unable to find a proof or a counterexample.

\item  What do the partial Desargues's theorems in Section \ref{sec:desargues} imply about automorphisms and coordinatizations?

\item  Is there a rigorous sense in which a projective rectangle is higher-dimensional, as suggested in Section \ref{sec:desargues} and \cite{pr4}?

\item  If every plane in \PR\ is Moufang, it has coordinates in an alternative ring.  If all such rings are isomorphic, does \PR\ extend to a Moufang plane with an alternative ring that extends that of the planes in \PR?

\item  Given a projective rectangle, in what projective planes can it be embedded?  In particular, our constructions by subplanes and harmonic extension give projective rectangles embedded in a Pappian plane but the same rectangles may possibly be isomorphically embeddable in planes that are not Pappian, not Desarguesian, maybe not even Moufang, in a nontrivial way, i.e., not by finding the Pappian plane as a subplane of a non-Pappian plane.  

\end{enumerate}

\sectionpage


\begin{thebibliography}{99}

\bibitem{dk} 
J.~D\'enes and A.~D.~Keedwell, \emph{Latin Squares and Their Applications}.  
Academic Press, New York--London, 1974.

\bibitem{dls}  Jeff H.~Dinitz, Alan~C.~H.~Ling, and Douglas~R.~Stinson, Perfect hash families from transversal designs.  \emph{Australas. J. Combin.} {\bf 37} (2007), 233--242.

\bibitem{rflc}
Rigoberto~Fl\'orez, Lindstr\"om's conjecture on a class of algebraically non-representable matroids. 
\emph{Europ. J. Combin.} {\bf 27} (2006), 896--905.

\bibitem{rfhc} Rigoberto~Fl\'orez, Harmonic conjugation in harmonic matroids. 
\emph{Discrete Math.} {\bf 309} (2009), 2365--2372.

\bibitem{bgpp}
Rigoberto~Fl\'orez and Thomas~Zaslavsky, Projective planarity of matroids of 3-nets and biased graphs. 
\emph{Australasian J.\ Combin.} {\bf 77}(2) (2020), 299--338.

\bibitem{pr3}
Rigoberto~Fl\'orez and Thomas~Zaslavsky, Projective rectangles: Harmonic conjugation.  Submitted.  arXiv:2406.12248

\bibitem{pr2a}
Rigoberto~Fl\'orez and Thomas~Zaslavsky, Projective rectangles: The graph of lines.  Submitted.

\bibitem{pr4}
Rigoberto~Fl\'orez and Thomas~Zaslavsky, Projective rectangles: The graph of planes and higher structure.  In preparation.

\bibitem{Hedayat} A.~S.~Hedayat, N.~J.~A.~Sloane, and J.~Stufken, 
\emph{Orthogonal Arrays, Theory and Applications}.
Springer-Verlag, New York, 1999.

\bibitem{HP}
Daniel R.~Hughes and Fred C.~Piper,  \emph{Projective Planes}.
Grad.\ Texts in Math., Vol.\ 6.  Springer-Verlag, New York, 1973.
MR 48 \#12278.  Zbl 267.50018.

\bibitem{ldt}
Bernt~Lindstr\"om, A Desarguesian theorem for algebraic combinatorial geometries. 
\emph{Combinatorica} {\bf 5} (1985), no. 3, 237--239.

\bibitem{lhc}
Bernt~Lindstr\"om, On harmonic conjugates in full algebraic combinatorial geometries. 
\emph{Europ. J. Combin.} {\bf 7} (1986), 259--262.

\bibitem{Ryser} H.~J.~Ryser, \emph{Combinatorial Mathematics}. 
Carus Math. Monographs, No. 14. 
Math. Assoc. Amer., New York, 1963.

\bibitem{st}
F.~W.~Stevenson, \emph{Projective Planes}. W.~H.~Freeman, San Francisco, 1972.

\bibitem{vw}  J.~H.~van~Lint and R.~M.~Wilson, \emph{A Course in Combinatorics}.  Second ed.  Cambridge University Press, Cambridge, Eng., 2001.

\bibitem{b1} Thomas Zaslavsky,
Biased graphs.  I.\  Bias, balance, and gains.  
\emph{J.\ Combin.\ Theory Ser.\ B} {\bf 47} (1989), 32--52.  

\bibitem{b2}
Thomas~Zaslavsky, Biased graphs. II. The three matroids. 
\emph{J. Combin. Theory Ser. B} {\bf 51} (1991), 46--72.

\bibitem{b4}
Thomas~Zaslavsky, Biased graphs IV: Geometrical realizations.
\emph{J. Combin. Theory Ser. B} {\bf 89} (2003), 231--297.

\end{thebibliography}
\end{document}